\documentclass{article}

\usepackage{fullpage}

\usepackage{amsmath}
\usepackage{amsthm}
\usepackage{amsfonts}
\usepackage{tikz-cd}
\usepackage[utf8x]{inputenc}
\usepackage{enumerate}
\usepackage[english]{babel}
\usepackage{amssymb}
\usepackage{esint}
\usepackage{hyperref}
\usepackage{bm}
\usepackage{bbm}

\theoremstyle{plain}
\newtheorem{theorem}{Theorem}[section]
\newtheorem{lemma}[theorem]{Lemma}

\theoremstyle{definition}
\newtheorem{definition}[theorem]{Definition}
\theoremstyle{remark}

\newtheorem{remark}[theorem]{Remark}

\DeclareMathOperator{\divv}{div}
\DeclareMathOperator*{\esssup}{ess\:sup}
\DeclareMathOperator{\trace}{trace}
\DeclareMathOperator*{\supp}{supp}
\DeclareMathOperator*{\meas}{meas}

\title{Vanishing viscosity limit for the compressible Navier--Stokes system\\via measure--valued solutions}
\author{Danica Basari\'{c}}
\date{}

\begin{document}
	\maketitle
	
	\begin{center}
		Technische Universit\"{a}t Berlin \\
		Institute f\"{u}r Mathematik, Stra{\ss}e des 17. Juni 136, 10623 Berlin, Germany 
	\end{center}

	\begin{center}
		E-mail address: basaric@math.tu-berlin.de
	\end{center}
	
	\begin{abstract}
		We identify a class of \textit{measure--valued} solutions of the barotropic Euler system 
		on a general (unbounded) spatial domain as a vanishing viscosity limit for the compressible Navier--Stokes system. Then we establish 
		the weak (measure--valued)--strong uniqueness principle, and, as a corollary, we obtain strong convergence to the Euler system on the lifespan of the strong solution.
	\end{abstract}
	
	We consider the compressible Euler system with damping 
	\begin{equation} \label{continuity equation in function of the momentum}
		\partial_t \varrho + \divv_x \textbf{m}=0,
	\end{equation}
	\begin{equation} \label{momentum equation in function of the momentum}
		\partial_t \textbf{m} + \divv_x \left(\frac{\textbf{m}\otimes \textbf{m}}{\varrho}\right) + \nabla_x p(\varrho) +a\textbf{m}=0;
	\end{equation}
	here $\varrho=\varrho(t,x)$ denotes the density, $\textbf{m}=\textbf{m}(t,x)$ the momentum - with the convection that the convective term is equal to zero whenever $\varrho=0$ - and $p=p(\varrho)$ the pressure. The term $a \textbf{m}$, with $a\geq0$, represents ``friction''. We will study the system on the set $(t,x) \in (0,T)\times \Omega$, where $T>0$ is a fixed time, $\Omega \subseteq \mathbb{R}^N$ with $N=2,3$, can be a bounded or unbounded domain, along with the boundary condition
	\begin{equation} \label{boundary condition}
		\textbf{m} \cdot \textbf{n}|_{\partial \Omega} = 0,
	\end{equation}
	for all $t\in [0,T]$; if $\Omega$ is unbounded, we impose the condition at infinity
	\begin{equation} \label{conditions at infinity}
		\varrho\rightarrow \overline{\varrho}, \quad \textbf{m}  \rightarrow 0 \quad \mbox{as } |x|\rightarrow \infty,
	\end{equation}
	with $\overline{\varrho}\geq 0$. We also consider the following initial data 
	\begin{equation} \label{initial conditions}
		\varrho(0,\cdot) =\varrho_0, \quad \textbf{m}(0,\cdot)=\textbf{m}_0,
	\end{equation}
	with $\varrho_0 >0$. We finally assume that the pressure $p$ is given by the isentropic state equation
	\begin{equation} \label{pressure}
		p(\varrho)=A \varrho^{\gamma},
	\end{equation}
	where $\gamma >1$ is the adiabatic exponent and $A>0$ is  a constant.
	
	Our goal is to identify a class of generalized - \emph{dissipative measure valued (DMV)} solutions - for the 
	Euler system (\ref{continuity equation in function of the momentum}), (\ref{momentum equation in function of the momentum})
	as a vanishing viscosity limit of the Navier--Stokes equations. More specifically, we start considering the set
	\begin{equation*}
		\Omega_R= \Omega \cap B_R, \quad B_R= \{ x\in \mathbb{R}^N: |x|<R \},
	\end{equation*}
	where we assume $\Omega_R$ to be at least a Lipschitz domain, and we consider the Navier--Stokes system:
	\begin{equation} \label{continuity equation N-S}
		\partial_t \varrho + \divv_x(\varrho \textbf{u}) =0,
	\end{equation}
	\begin{equation} \label{balance of momentum N-S}
		\partial_t(\varrho \textbf{u}) + \divv_x (\varrho \textbf{u}\otimes \textbf{u}) +\nabla_x p(\varrho)= \frac{1}{R} \divv_x \mathbb{S}(\nabla_x \textbf{u})-a\varrho \textbf{u};
	\end{equation}
	now $\textbf{u}=\textbf{u}(t,x)$ is the velocity and $\mathbb{S}= \mathbb{S}(\nabla_x \textbf{u})$ is the viscous stress, which we assume to be a linear function of the velocity gradient, more specifically to satisfy the Newton's rheological law 
	\begin{equation} \label{viscous stress}
		\mathbb{S}=\mathbb{S}(\nabla_x \textbf{u})= \mu \left( \nabla_x \textbf{u} + \nabla_x^T \textbf{u} -\frac{2}{N} (\divv_x \textbf{u}) \mathbb{I}\right)+ \eta(\divv_x \textbf{u})\mathbb{I},
	\end{equation}
	where $\mu >0$, $\eta \geq 0$ are constants. Introducing $\lambda=\eta -\frac{2}{N}\mu$ we also have
	\begin{equation} \label{alternative versione of the viscous stress}
		\mathbb{S}(\nabla_x \textbf{u})= \mu (\nabla_x \textbf{u} + \nabla_x^T \textbf{u}) + \lambda (\divv_x \textbf{u})\mathbb{I}, \quad \mu >0, \lambda \geq -\frac{2}{N}\mu.
	\end{equation}
	
	As our goal is to perform the vanishing viscosity limit for the Navier--Stokes system, we impose the complete slip boundary conditions on $\partial \Omega$:
	\begin{equation} \label{compl}
		\textbf{u} \cdot \textbf{n}|_{\partial \Omega} = 0,\ (\mathbb{S} \cdot \textbf{n}) \times \textbf{n}|_{\partial \Omega} = 0,
	\end{equation}
	and the no--slip boundary conditions on $\partial B_R$:
	\begin{equation} \label{boundary condition N-S}
		\textbf{u}|_{\partial B_R}=0,
	\end{equation}
	for all $t\in [0,T]$. Of course, \eqref{compl} and \eqref{boundary condition N-S} are compatible only if 
	$\partial B_R \cap \partial \Omega = \emptyset$ for $R$ large enough meaning that $\partial \Omega$ is a compact set. That is $\Omega$ is either (i) bounded, or (ii) exterior domain, or (iii) $\Omega = R^N$. For the sake of simplicity, we restrict ourselves to these three cases.
	
	Finally, we impose the initial conditions:
	\begin{equation} \label{initial conditions N-S}
		\varrho(0,\cdot) =\varrho_0, \quad (\varrho\textbf{u})(0,\cdot)= (\varrho\textbf{u})_0 \quad \mbox{in }\Omega_R.
	\end{equation}
	
	Our goal will be first to show that the solutions of the Navier--Stokes system converge to the measure-valued solution of the Euler system with damping in the zero viscosity limit, then we will prove the weak-strong uniqueness principle for the Euler system, see \hyperref[theorem weak-strong uniqueness]{Theorem \ref*{theorem weak-strong uniqueness}}. Then we conclude that solutions of the Navier--Stokes system converge to smooth solution of the Euler system as long as the latter exists, see \hyperref[Thvvl]{Theorem \ref*{Thvvl}}. Note that the vanishing viscosity limit for the compressible Navier--Stokes system on a bounded domain was studied by Sueur \cite{Sue1}. Our goal is to propose an alternative approach based on the concept of dissipative measure--valued solutions and extend the result to a more general class of domains. The concept of dissipative measure--valued solution is of independent interest and has been use recently in the analysis of convergence of certain numerical schemes, 
	see \cite{FeiLukMiz}.
	
	\section{From the Navier--Stokes to the Euler system}
	\subsection{Weak formulation}
	
	To get the weak formulation of the Navier--Stokes system, we simply multiply both equations \eqref{continuity equation N-S}, \eqref{balance of momentum N-S} by test functions, and, supposing also that the density $\varrho$ and the momentum $\varrho \textbf{u}$ are weakly continuous in time, we get
	\begin{equation} \label{weak formulation continuity equation N-S}
		\left[\int_{\Omega_R} \varrho\varphi(t,\cdot) dx\right]_{t=0}^{t=\tau} = \int_{0}^{\tau} \int_{\Omega_R} [\varrho\partial_t\varphi +\varrho\textbf{u}\cdot \nabla_x \varphi] dxdt,
	\end{equation}
	for any $\tau \in [0,T)$ and all $\varphi \in C^1_c([0,T]\times \overline{\Omega}_R)$, and
	\begin{equation} \label{weak formulation momentum balance N-S}
		\left[\int_{\Omega_R} \varrho\textbf{u} \cdot \bm{\varphi}(t,\cdot) dx \right]_{t=0}^{t=\tau} = \int_{0}^{\tau} \int_{\Omega_R} \left[ \varrho\textbf{u} \cdot \partial_t \bm{\varphi} +\varrho \textbf{u}\otimes \textbf{u}: \nabla_x \bm{\varphi} + p(\varrho) \mbox{div}_x \bm{\varphi} - \frac{1}{R} \mathbb{S}(\nabla_x \textbf{u}): \nabla_x \bm{\varphi}- a\varrho\textbf{u}\cdot \bm{\varphi}\right] dx dt,
	\end{equation}
	for any $\tau \in [0,T)$ and all $\bm{\varphi} \in C^1_c([0,T]\times \overline{\Omega} \cap B_R; \mathbb{R}^N)$ with $\bm{\varphi}\cdot \textbf{n}|_{\partial \Omega}=0$.
	
	Multiplying \eqref{balance of momentum N-S} by $\textbf{u}$ and introducing the \textit{pressure potential} $P$ as the solution of the equation
	\begin{equation*}
		\varrho P'(\varrho) -P(\varrho)=p(\varrho),
	\end{equation*}
	which, for instance, in our case can be taken as
	\begin{equation*}
		P(\varrho)= \varrho \int_{\overline{\varrho}}^{\varrho} \frac{p(z)}{z^2} dz,
	\end{equation*}
	(notice in particular that $P(\overline{\varrho})=0$; this will be used later) we get the energy equality
	\begin{equation} \label{energy equality N-S}
		\frac{d}{dt} \int_{\Omega_R} \left[ \frac{1}{2}\varrho |\textbf{u}|^2 +  P(\varrho)\right] dx + a \int_{\Omega_R}\varrho |\textbf{u}|^2 dx + \frac{1}{R}\int_{\Omega_R} \mathbb{S}(\nabla_x \textbf{u}): \nabla_x \textbf{u} dx =0,
	\end{equation}
	from which the \textit{energy inequality} follows
	\begin{align} \label{energy inequality N-S}
		\int_{\Omega_R} \left[ \frac{1}{2}\varrho |\textbf{u}|^2 +  P(\varrho)\right] (\tau,\cdot) dx + a \int_{0}^{\tau} \int_{\Omega_R}\varrho |\textbf{u}|^2 dx dt &+\frac{1}{R} \int_{0}^{\tau} \int_{\Omega_R} \mathbb{S}(\nabla_x \textbf{u}): \nabla_x \textbf{u} dxdt \\
		&\leq \int_{\Omega_R} \left[ \frac{1}{2\varrho_0}|(\varrho\textbf{u})_0|^2 + P(\varrho_0)\right] dx,
	\end{align}
	for a.e. $\tau \in [0,T]$. For more details see \cite{FeiKarPok}.
	
	\subsection{Existence of weak solutions}
	Now, we have the following result (for the proof see \cite{FeiKarPok}).
	\begin{theorem}
		Let $\Omega_R \subset \mathbb{R}^N$ be a Lipschitz domain with compact boundary $\Omega_R = \Omega \cap B_R$, $\partial \Omega \cap \partial B_R = \emptyset$, and let $T>0$ be arbitrary. Suppose that the initial data satisfy
		\begin{equation*}
			\varrho_0 \in L^{\gamma}(\Omega_R), \quad \varrho_0\geq 0\mbox{ a.e. in } \Omega_R, \quad \frac{|(\varrho \textbf{u})_0|^2}{\varrho_0} \in L^1(\Omega_R).
		\end{equation*}
		Let the pressure $p$ satisfy \eqref{pressure} with
		\begin{equation*}
			\gamma > \frac{N}{2}.
		\end{equation*}
		Then the Navier--Stokes system \eqref{continuity equation N-S}-\eqref{initial conditions N-S} admits a weak solution $[\varrho, \textbf{u}]$ in $(0,T)\times \Omega_R$ such that
		\begin{enumerate}
			\item the density $\varrho=\varrho(t,x)$ is a non-negative function a.e. in $(0,T)\times \Omega_R$ and satisfies
			\begin{equation*}
				\varrho \in C_{weak}([0,T]; L^{\gamma}(\Omega_R));
			\end{equation*}
			the velocity $\textbf{u}=\textbf{u}(t,x)$ satisfies
			\begin{equation*}
				\textbf{u}\in L^2(0,T; W^{1,2}(\Omega_R; \mathbb{R}^N)), 
				\ \textbf{u} \cdot \textbf{n}|_{\partial \Omega} = 0, \ \textbf{u}|_{\partial B_R} = 0 ;
			\end{equation*}
			the momentum $\varrho \textbf{u}=(\varrho \textbf{u})(t,x)$ satisfies
			\begin{equation*}
				\varrho \textbf{u} \in C_{weak}([0,T];L^{\frac{2\gamma}{\gamma+1}}(\Omega_R;\mathbb{R}^N));
			\end{equation*}
			\item the weak formulations of the continuity equation \eqref{weak formulation continuity equation N-S} and of the momentum balance \eqref{weak formulation momentum balance N-S} are satisfied in $(0,T)\times \Omega_R$;
			\item the energy inequality \eqref{energy inequality N-S} holds for a.e. $\tau \in [0,T]$.
		\end{enumerate}
	\end{theorem}

	\subsection{Reformulation of the problem in terms of a background density $\overline{\varrho}$}
	
	Choosing a background density $\overline{\varrho}\geq 0$, we can slightly change the energy inequality; indeed the Navier--Stokes system can be rewritten as
	\begin{equation*}
		\partial_t(\varrho - \overline{\varrho}) + \divv_x(\varrho \textbf{u})=0,
	\end{equation*}
	\begin{equation*}
		\partial_t(\varrho \textbf{u}) + \divv_x (\varrho \textbf{u}\otimes \textbf{u}) +\nabla_x[p(\varrho)-p(\overline{\varrho})] = \frac{1}{R} \divv_x \mathbb{S}(\nabla_x \textbf{u})-a\varrho \textbf{u}.
	\end{equation*}
	Again, multiplying both equations by test functions and using the weak continuity in time we obtain 
	\begin{equation*}
		\left[\int_{\Omega_R} (\varrho-\overline{\varrho})\varphi(t,\cdot) dx\right]_{t=0}^{t=\tau} = \int_{0}^{\tau} \int_{\Omega_R} [(\varrho-\overline{\varrho})\partial_t\varphi +\varrho\textbf{u}\cdot \nabla_x \varphi] dxdt,
	\end{equation*}
	for any $\tau \in [0,T)$ and all $\varphi \in C^1_c([0,T]\times \overline{\Omega}_R)$, and
	\begin{equation*}
		\left[\int_{\Omega_R} \varrho\textbf{u} \cdot \bm{\varphi}(t,\cdot) dx \right]_{t=0}^{t=\tau} = \int_{0}^{\tau} \int_{\Omega_R} \left[ \varrho\textbf{u} \cdot \partial_t \bm{\varphi} +\varrho \textbf{u}\otimes \textbf{u}: \nabla_x \bm{\varphi} + [p(\varrho)-p(\overline{\varrho})] \mbox{div}_x \bm{\varphi} - \frac{1}{R} \mathbb{S}(\nabla_x \textbf{u}): \nabla_x \bm{\varphi}- a\varrho\textbf{u}\cdot \bm{\varphi}\right] dx dt,
	\end{equation*}
	for any $\tau \in [0,T)$ and all $\bm{\varphi} \in C^1_c([0,T]\times \overline{\Omega} \cap B_R; \mathbb{R}^N)$ with $\bm{\varphi}\cdot \textbf{n}|_{\partial \Omega}=0$.
	Also, integrating the first equation over $\Omega_R$ along with condition \eqref{boundary condition N-S}, we get
	\begin{equation*}
		\frac{d}{dt} \int_{\Omega_R} (\varrho-\overline{\varrho}) dx=0 \quad \Rightarrow \quad  \frac{d}{dt} \int_{\Omega_R} P'(\overline{\varrho}) (\varrho-\overline{\varrho}) dx=0.
	\end{equation*}
	Since $P(\overline{\varrho})=0$, we can rewrite \eqref{energy equality N-S} as
	\begin{equation*}
		\frac{d}{dt} \int_{\Omega_R} \left[ \frac{1}{2}\varrho |\textbf{u}|^2 +  P(\varrho) -P'(\overline{\varrho})(\varrho- \overline{\varrho})- P(\overline{\varrho})\right] dx + a \int_{\Omega_R}\varrho |\textbf{u}|^2 dx + \frac{1}{R}\int_{\Omega_R} \mathbb{S}(\nabla_x \textbf{u}): \nabla_x \textbf{u} dx =0,
	\end{equation*}
	and \eqref{energy inequality N-S} becomes
	\begin{align*}
		\int_{\Omega_R} \left[ \frac{1}{2}\varrho |\textbf{u}|^2 +  P(\varrho)-P'(\overline{\varrho})(\varrho- \overline{\varrho})- P(\overline{\varrho})\right] (\tau,\cdot) dx &+ a \int_{0}^{\tau} \int_{\Omega_R}\varrho |\textbf{u}|^2 dx dt + \frac{1}{R}\int_{0}^{\tau} \int_{\Omega_R} \mathbb{S}(\nabla_x \textbf{u}): \nabla_x \textbf{u} dxdt \\
		&\leq \int_{\Omega_R} \left[ \frac{1}{2\varrho_0}|(\varrho\textbf{u})_0|^2 + P(\varrho_0)-P'(\overline{\varrho})(\varrho_0- \overline{\varrho})- P(\overline{\varrho})\right] dx,
	\end{align*}
	for a.e. $\tau \in [0,T]$.
	
	\subsection{Weak sequential stability} \label{WSS}
	We can then consider the family $\{ \varrho_R-\overline{\varrho}, \textbf{m}_R= \varrho_R \textbf{u}_R\}_{R>0}$ of dissipative weak solutions to the previous Navier--Stokes system with initial data $\{ \varrho_{R,0}-\overline{\varrho}, \textbf{m}_{R,0} \}_{R>0}$ defined in $(0,T)\times \Omega$; in particular, they will satisfy the following conditions:
	\begin{equation} \label{weak formulation continuity equation Navier-Stokes}
		\int_{0}^{T} \int_{\Omega} [(\varrho_R-\overline{\varrho})\partial_t\varphi +\varrho_R\textbf{u}_R\cdot \nabla_x \varphi] dxdt =0,
	\end{equation}
	for all $\varphi \in C^1_c((0,T)\times \Omega)$, and
	\begin{equation} \label{weak formulation balance of momentum Navier-Stokes}
		\int_{0}^{T} \int_{\Omega} \left[ \varrho_R\textbf{u}_R \cdot \partial_t \bm{\varphi} +\varrho_R \textbf{u}_R\otimes \textbf{u}_R: \nabla_x \bm{\varphi} + [p(\varrho_R)-p(\overline{\varrho})] \mbox{div}_x \bm{\varphi} - \frac{1}{R} \mathbb{S}(\nabla_x \textbf{u}_R): \nabla_x \bm{\varphi}- a\varrho_R\textbf{u}_R\cdot \bm{\varphi} \right] dx dt =0
	\end{equation}
	for all $\bm{\varphi} \in C^1_c((0,T)\times \Omega; \mathbb{R}^N)$. We have replaced $\Omega_R$ by $\Omega$ in the previous integrals. Note that this is correct for $R$ large enough as the test functions are compactly supported in $\Omega_R$. 
	
	More precisely, thanks to the weak continuity of the densities and momenta, we have 
	\begin{equation*}
		\left[\int_{\Omega} (\varrho_R-\overline{\varrho})\varphi(t,\cdot) 	dx\right]_{t=0}^{t=\tau} = \int_{0}^{\tau} \int_{\Omega} [(\varrho_R-\overline{\varrho})\partial_t\varphi +\varrho_R\textbf{u}_R\cdot \nabla_x \varphi ] dxdt,
	\end{equation*}
	for any $\tau \in [0,T)$ and all $\varphi \in C^1_c([0,T]\times \overline{\Omega})$,
	\begin{align*} 
		\left[\int_{\Omega} \varrho_R\textbf{u}_R \cdot \bm{\varphi}(t,\cdot) dx \right]_{t=0}^{t=\tau} &= \int_{0}^{\tau} \int_{\Omega} \left[ \varrho_R\textbf{u}_R \cdot \partial_t \bm{\varphi} +\varrho_R \textbf{u}_R\otimes \textbf{u}_R: \nabla_x \bm{\varphi} + (p(\varrho_R)-p(\overline{\varrho})
		) \mbox{div}_x \bm{\varphi} \right] dx dt, \\
		& -\int_{0}^{\tau} \int_{\Omega} \left[ \frac{1}{R} \mathbb{S}(\nabla_x \textbf{u}_R): \nabla_x \bm{\varphi}+ a\varrho_R\textbf{u}_R\cdot \bm{\varphi} \right] dxdt
	\end{align*}
	for any $\tau \in [0,T)$ and all $\bm{\varphi} \in C^1_c([0,T]\times \overline{\Omega}; \mathbb{R}^N)$,
	$\bm{\varphi} \cdot \textbf{n}|_{\partial \Omega} = 0$.
	
	Finally, we have the energy inequality
	\begin{equation} \label{energy inequality}
		\begin{aligned}
			\int_{\Omega} \left[ \frac{1}{2}\varrho_R |\textbf{u}_R|^2 +  P(\varrho_R)-P'(\overline{\varrho})(\varrho_R- \overline{\varrho})- P(\overline{\varrho})\right] &(\tau,\cdot)dx + a \int_{0}^{\tau} \int_{\Omega}\varrho_R |\textbf{u}_R|^2 dx dt + \frac{1}{R}\int_{0}^{\tau} \int_{\Omega} \mathbb{S}(\nabla_x \textbf{u}_R): \nabla_x \textbf{u}_R dxdt \\
			&\leq \int_{\Omega} \left[ \frac{1}{2\varrho_{R,0}}|(\varrho\textbf{u})_{R,0}|^2 + P(\varrho_{R,0})-P'(\overline{\varrho})(\varrho_{R,0}- \overline{\varrho})- P(\overline{\varrho})\right] dx,
		\end{aligned}
	\end{equation}
	for a.e. $\tau \in [0,T]$. In \eqref{energy inequality}, we suppose that the initial data have been chosen on $\Omega$ in such a way that
	\begin{equation} \label{bound on the initial energy}
		\int_{\Omega} \left[ \frac{1}{2\varrho_{R,0}}|(\varrho\textbf{u})_{R,0}|^2 + P(\varrho_{R,0})-P'(\overline{\varrho})(\varrho_{R,0}- \overline{\varrho})- P(\overline{\varrho})\right] dx \leq E_0,
	\end{equation}
	where the constant $E_0$ is independent of $R$. Then, extending $\textbf{u}_R$ to be zero and $\varrho_R$ as $\overline{\varrho}$ outside $B_R$, we easily deduce from the energy inequality that
	\begin{equation} \label{bound momentum sqrt density}
		\esssup_{t\in(0,T)} \| \sqrt{\varrho_R} \textbf{u}_R(t,\cdot) \|_{L^2(\Omega;\mathbb{R}^N)} \leq c(E_0),
	\end{equation}
	\begin{equation} \label{potential bound}
		\esssup_{t\in(0,T)} \| (P(\varrho_R)-P'(\overline{\varrho})(\varrho_R- \overline{\varrho})- P(\overline{\varrho}))(t,\cdot) \|_{L^1(\Omega)} \leq c(E_0),
	\end{equation}
	\begin{equation} \label{viscous stress estimate}
		\frac{1}{R}\int_{0}^{T} \int_{\Omega} \mathbb{S}(\nabla_x \textbf{u}_R): \nabla_x \textbf{u}_R dxdt \leq c(E_0),
	\end{equation}
	where the bounds are independent of $R$. Next, from \eqref{viscous stress estimate}, we can deduce that
	\begin{equation} \label{bound on the terms of the viscous stress tensor}
		\frac{1}{R}\int_{0}^{T} \| \mathbb{S} (\nabla_x \textbf{u}_R) (t,\cdot) \|_{L^2(\Omega; \mathbb{R}^N \times \mathbb{R}^N)}^2 dt \leq c(E_0).
	\end{equation}
	
	Now, we can use the following relation
	\begin{equation*}
		P(\varrho)-P'(\overline{\varrho})(\varrho-\overline{\varrho})-P(\overline{\varrho}) \geq c(\overline{\varrho}) \begin{cases}
			(\varrho-\overline{\varrho})^2 \quad &\mbox{for } \frac{\overline{\varrho}}{2} < \varrho < 2\overline{\varrho} \\
			(1+\varrho^{\gamma}) \quad &\mbox{otherwise},
		\end{cases}
	\end{equation*}
	for a positive constant $c(\overline{\varrho})$ (see \cite{FeiJinNov}). Following \cite{FeiNov}, we introduce the decomposition of an integrable 
	function $h_R$:
	\begin{equation*}
		h_R=[h_R]_{ess} + [h_R]_{res},
	\end{equation*}
	where
	\begin{equation*}
		[h_R]_{ess}=\chi(\varrho_R)h_R, \quad [h_R]_{res}= (1-\chi(\varrho_R)) h_R,
	\end{equation*}
	\begin{equation*}
		\chi \in C_c^{\infty}(0,\infty), \mbox{ } 0\leq \chi \leq 1, \mbox{ } \chi(r)=1 \mbox{ for } r\in \left[ \frac{\overline{\varrho}}{2},2\overline{\varrho} \right].
	\end{equation*}
	Then we have
	\begin{equation} \label{density bound in L2}
		\begin{aligned}
			\esssup_{t\in(0,T)} \|[\varrho_R-\overline{\varrho}]_{ess}(t,\cdot)\|_{L^2(\Omega)} &= \esssup_{t\in(0,T)} \int_{\Omega} (\varrho_R-\overline{\varrho})^2\chi(\varrho_R)(t,\cdot)  dx \\
			&\leq \frac{1}{c(\overline{\varrho})} \esssup_{t\in(0,T)} \| (P(\varrho_R)-P'(\overline{\varrho})(\varrho_R- \overline{\varrho})- P(\overline{\varrho}))(t,\cdot) \|_{L^1(\Omega)} \\
			&\leq c(E_0)
		\end{aligned}
	\end{equation}
	and
	\begin{equation} \label{density bound in Lgamma}
		\begin{aligned}
			\esssup_{t\in(0,T)} \|[\varrho_R-\overline{\varrho}]_{res}(t,\cdot)\|_{L^{\gamma}(\Omega)} &= \esssup_{t\in(0,T)} \int_{\Omega} |\varrho_R-\overline{\varrho}|^{\gamma}(1-\chi(\varrho_R))(t,\cdot)  dx \\
			&\lesssim \esssup_{t\in(0,T)} \int_{\Omega} (1+\varrho^{\gamma})(1-\chi(\varrho_R))(t,\cdot)  dx \\
			&\leq \frac{1}{c(\overline{\varrho})} \esssup_{t\in(0,T)} \| (P(\varrho_R)-P'(\overline{\varrho})(\varrho_R- \overline{\varrho})- P(\overline{\varrho}))(t,\cdot) \|_{L^1(\Omega)} \\
			&\leq c(E_0),
		\end{aligned}
	\end{equation}
	where $\lesssim$ means modulo a multiplication constant. In particular this implies that
	\begin{equation*}
		[\varrho_R-\overline{\varrho}]_{ess} \overset{*}{\rightharpoonup} f_{\varrho_R-\overline{\varrho}} \quad \mbox{in } L^{\infty}(0,T; L^2(\Omega)),
	\end{equation*}
	\begin{equation*}
		[\varrho_R-\overline{\varrho}]_{res} \overset{*}{\rightharpoonup} g_{\varrho_R-\overline{\varrho}} \quad \mbox{in } L^{\infty}(0,T; L^{\gamma}(\Omega));
	\end{equation*}
	passing to suitable subsequences as the case may be; defining $\varrho-\overline{\varrho}:=f_{\varrho_R-\overline{\varrho}}+g_{\varrho_R-\overline{\varrho}}$, we have that
	\begin{equation*}
		\varrho_R -\overline{\varrho} \overset{*}{\rightharpoonup} \varrho -\overline{\varrho} \quad \mbox{in }L^{\infty}(0,T; L^2+L^{\gamma}(\Omega)).
	\end{equation*}
	We can repeat the same procedure for the momenta; indeed, using \eqref{bound momentum sqrt density} we have
	\begin{equation} \label{momentum bound in L2}
		\begin{aligned}
			\esssup_{t\in(0,T)} \|[\varrho_R\textbf{u}_R]_{ess}(t,\cdot)\|_{L^2(\Omega)} &= \esssup_{t\in(0,T)} \int_{\Omega} \varrho_R \cdot \varrho_R|\textbf{u}_R|^2\chi(\varrho_R)(t,\cdot)  dx \\
			&\leq 2\overline{\varrho} \mbox{ } \esssup_{t\in(0,T)} \|\sqrt{\varrho_R} \textbf{u}_R(t,\cdot)\|_{L^2(\Omega)} \\
			&\leq c(E_0);
		\end{aligned}
	\end{equation}
	we also have
	\begin{align*}
		\esssup_{t\in(0,T)} \| [\sqrt{\varrho_R}]_{res}(t,\cdot)\|_{L^{2\gamma}(\Omega)} &= \esssup_{t\in(0,T)} \int_{\Omega} \varrho_R^{\gamma} (1-\chi(\varrho_R)) (t,\cdot) dx \\
		& \leq \esssup_{t\in(0,T)} \int_{\Omega} (\varrho_R^{\gamma}+1) (1-\chi(\varrho_R)) (t,\cdot) dx\\
		&\leq c(E_0),
	\end{align*}
	which, together with \eqref{bound momentum sqrt density} and H\"{o}lder's inequality with $p=\gamma+1$, gives
	\begin{equation} \label{momentum bound in Lq}
		\begin{aligned}
			\esssup_{t\in(0,T)} \|[\varrho_R\textbf{u}_R]_{res}(t,\cdot)\|_{L^{\frac{2\gamma}{\gamma+1}}(\Omega)} \leq \esssup_{t\in(0,T)}  \| [\sqrt{\varrho_R}]_{res}(t,\cdot)\|_{L^{2\gamma}(\Omega)} \|\sqrt{\varrho_R} \textbf{u}_R(t,\cdot)\|_{L^2(\Omega)} \leq c(E_0).
		\end{aligned}
	\end{equation}
	Then we obtain
	\begin{equation*}
		\varrho_R\textbf{u}_R \overset{*}{\rightharpoonup} \textbf{m} \quad \mbox{in } L^{\infty}(0,T; L^2+ L^{\frac{2\gamma}{\gamma+1}}(\Omega)),
	\end{equation*}
	passing to suitable subsequences as the case may be. In a similar way we have
	\begin{align*}
		\esssup_{t\in(0,T)} \|[p(\varrho_R)-p(\overline{\varrho})]_{ess}(t,\cdot)\|_{L^2(\Omega)} &= \esssup_{t\in(0,T)} \int_{\Omega} |p(\varrho_R)-p(\overline{\varrho})|^2 \chi(\varrho_R)(t,\cdot) dx \\
		&\leq p'(2\overline{\varrho}) \mbox{ }	\esssup_{t\in(0,T)} \|[\varrho_R-\overline{\varrho}]_{ess}(t,\cdot)\|_{L^2(\Omega)} \\
		&\leq c(E_0),
	\end{align*}
	and
	\begin{align*}
		\esssup_{t\in(0,T)} \|[p(\varrho_R)-p(\overline{\varrho})]_{res}(t,\cdot)\|_{L^1(\Omega)} &= A \esssup_{t\in(0,T)}\int_{\Omega} |\varrho_R^{\gamma}-\overline{\varrho}^{\gamma}| (1-\chi(\varrho_R)) (t,\cdot) dx \\
		& \leq A\max\{\overline{\varrho}^{\gamma},1\} \esssup_{t\in(0,T)} \int_{\Omega} (1+\varrho_R^{\gamma}) (1-\chi(\varrho_R)) (t,\cdot) dx \\
		&\leq c(E_0).
	\end{align*}
	Also, noticing that
	\begin{equation*}
		|\varrho_R \textbf{u}_R \otimes \textbf{u}_R| \lesssim \frac{1}{2} \varrho_R |\textbf{u}_R|^2,
	\end{equation*}
	from \eqref{bound momentum sqrt density} we deduce that also the convective terms are uniformly bounded in the non-reflexive space $L^1((0,T)\times \Omega)$, or better, in $L^{\infty}(0,T; L^1(\Omega))$. 
	
	There are two disturbing phenomena that may occur to bounded sequences in $L^1$: \textit{oscillations} and \textit{concentrations}. The idea is then to see $L^1((0,T)\times \Omega)$ as embedded in the space of bounded Radon measures $\mathcal{M}([0,T]\times \overline{\Omega})$ - that happens to be the dual to the separable space $C_0([0,T]\times \overline{\Omega})= \overline{C_c([0,T]\times \overline{\Omega}}^{\|\cdot\|_{\infty}}$ - through the identification
	\begin{equation*}
		\mu_f(\varphi)= \int_{0}^{T} \int_{\Omega} f\varphi dxdt, \quad \mbox{for all }\varphi \in C_0([0,T]\times \overline{\Omega}),
	\end{equation*}
	if $f\in L^1((0,T)\times \Omega)$.
	
	Accordingly, we may assume
	\begin{align*}
		\varrho_R -\overline{\varrho} \overset{*}{\rightharpoonup} \varrho -\overline{\varrho} & \quad \mbox{in }L^{\infty}(0,T; L^2+L^{\gamma}(\Omega));\\
		\varrho_R\textbf{u}_R \overset{*}{\rightharpoonup} \textbf{m} & \quad \mbox{in } L^{\infty}(0,T; L^2+ L^{\frac{2\gamma}{\gamma+1}}(\Omega)); \\
		\mu_{p(\varrho_R)-p(\overline{\varrho})} \overset{*}{\rightharpoonup} \mu_{\{p\}-p(\overline{\varrho})} & \quad \mbox{in }\mathcal{M}((0,T)\times \Omega); \\
		\mu_{\varrho_R \textbf{u}_R \otimes \textbf{u}_R} \overset{*}{\rightharpoonup} \mu_{\{ \mathbb{M} \}} & \quad \mbox{in }\mathcal{M}((0,T)\times \Omega; \mathbb{R}^N\times \mathbb{R}^N); \\
		\mu_{\frac{1}{2} \varrho_R |\textbf{u}|^2 + P(\varrho_R)-P'(\overline{\varrho})(\varrho_R-\overline{\varrho})-P(\overline{\varrho})} \overset{*}{\rightharpoonup} \mu_{\{E\}} & \quad \mbox{in }\mathcal{M}((0,T)\times \Omega),
	\end{align*}
	passing to suitable subsequences as the case may be. This means,
	\begin{equation*}
		\langle \mu_{p(\varrho_R)-p(\overline{\varrho})}; \varphi \rangle = \int_{0}^{T} \int_{\Omega} [p(\varrho_R)-p(\overline{\varrho})] \varphi dxdt \rightarrow \int_{0}^{T} \int_{\overleftarrow{\Omega}} [\{p\}-p(\overline{\varrho})] \varphi dxdt= \langle \mu_{\{p\}-p(\overline{\varrho})}; \varphi \rangle, \quad \mbox{as } R \rightarrow \infty, 
	\end{equation*}
	for every $\varphi \in C_c([0,T]\times \overline{\Omega})$; the same holds for the other convergences.
	
	We can now let $R\rightarrow \infty$ in \eqref{weak formulation continuity equation Navier-Stokes}, \eqref{weak formulation balance of momentum Navier-Stokes}; notice that the $R$-dependent viscous stress tensor vanishes. Indeed, using \eqref{bound on the terms of the viscous stress tensor} and H\"older's inequality we get
	\begin{equation*}
		\left| \frac{1}{R} \int_{0}^{T} \int_{\Omega} \mathbb{S}(\nabla_x \textbf{u}_R): \nabla_x \bm{\varphi} \ dxdt \ \right| \leq \frac{1}{\sqrt{R}} \left\| \frac{1}{\sqrt{R}}  \mathbb{S}(\nabla_x \textbf{u}_R) \right\|_{L^2((0,T) \times \Omega)}
		\left\| \nabla_x \bm{\varphi} \right\|_{L^2((0,T) \times \Omega)} \leq \frac{c(E_0)}{\sqrt{R}} \left\| \nabla_x \bm{\varphi} \right\|_{L^2((0,T) \times \Omega)}.
	\end{equation*}
	Then we get 
	\begin{equation*} 
		\int_{0}^{T} \int_{\Omega} [(\varrho-\overline{\varrho})\partial_t \varphi + \textbf{m}\cdot \nabla_x \varphi] dxdt=0,
	\end{equation*}
	for every $\varphi \in C_c^1((0,T)\times \Omega)$, and
	\begin{equation*} 
		\int_{0}^{T} \int_{\Omega} \left[ \textbf{m} \cdot \partial_t \bm{\varphi} +\{\mathbb{M}\} : \nabla_x \bm{\varphi} + (\{p\}-p(\overline{\varrho})) \mbox{div}_x \bm{\varphi} - a\textbf{m}\cdot \bm{\varphi}\right] dx dt=0, 
	\end{equation*}
	for every $\bm{\varphi}\in C_c^1((0,T)\times\Omega; \mathbb{R}^N)$. We can equivalently write
	\begin{equation*} 
		\int_{0}^{T} \int_{\Omega} [\varrho\partial_t \varphi + \textbf{m}\cdot \nabla_x \varphi] dxdt=0,
	\end{equation*}
	for every $\varphi \in C_c^1((0,T)\times \Omega)$, and
	\begin{equation*} 
		\int_{0}^{T} \int_{\Omega} \left[ \textbf{m} \cdot \partial_t \bm{\varphi} +\{\mathbb{M}\} : \nabla_x \bm{\varphi} + \{p\} \mbox{div}_x \bm{\varphi} - a\textbf{m}\cdot \bm{\varphi}\right] dx dt=0, 
	\end{equation*}
	for every $\bm{\varphi}\in C_c^1((0,T)\times\Omega; \mathbb{R}^N)$. As a matter of fact, the limit for $\varrho_R-\overline{\varrho}$ can be strengthened to
	\begin{equation*}
		\varrho_R-\overline{\varrho} \rightarrow \varrho-\overline{\varrho} \mbox{ in } C_{weak} ([0,T]; L^2+ L^{\gamma}(\Omega));
	\end{equation*}
	the same holds for the limit of $\varrho_R\textbf{u}_R$:
	\begin{equation*}
		\varrho_R\textbf{u}_R \rightarrow \textbf{m} \mbox{ in } C_{weak} ([0,T]; L^2+ L^{\frac{2\gamma}{\gamma+1}}(\Omega; \mathbb{R}^N)).
	\end{equation*}
	We can then rewrite the last two integral equations as
	\begin{equation} \label{eqcont}
		\int_{\Omega} \varrho \varphi (\tau, \cdot) dx - \int_{\Omega} \varrho \varphi(0,\cdot) dx = \int_{0}^{\tau} \int_{\Omega} [\varrho \partial_t \varphi+ \textbf{m}\cdot \nabla_x \varphi] dxdt,
	\end{equation}
	for any $\tau \in [0,T)$ and any $\varphi \in C_c^1([0,T]\times \overline{\Omega})$ and
	\begin{equation} \label{momentum equation with measures}
		\int_{\Omega} \textbf{m} \cdot \bm{\varphi}(\tau, \cdot) dx - \int_{\Omega} \textbf{m} \cdot \bm{\varphi}(0,\cdot) dx = \int_{0}^{\tau} \int_{\Omega} \left[ \textbf{m} \cdot \partial_t \bm{\varphi} +\{\mathbb{M}\} : \nabla_x \bm{\varphi} + \{p\} \mbox{div}_x \bm{\varphi} - a\textbf{m}\cdot \bm{\varphi}\right] dx dt,
	\end{equation}
	for any $\tau \in [0,T)$ and any $\bm{\varphi} \in C_c^1([0,T]\times \overline{\Omega};\mathbb{R}^N)$, $\bm{\varphi}\cdot \textbf{n}|_{\partial \Omega}=0$.
	
	Finally, using the generalization of the concept of Lebesgue point to Radon measures, we can deduce from the energy inequality \eqref{energy inequality}
	\begin{equation} \label{energy inequality with measure}
		\int_{\Omega}\{E\}(\tau, \cdot) dx + a \int_{0}^{\tau} \int_{\Omega} \trace \{\mathbb{M}\}dxdt \leq \int_{\Omega} \{E\}(0, \cdot) dx,
	\end{equation}
	for a.e. $\tau \in (0,T)$, where
	\begin{equation*}
		\int_{\Omega}\{E\}(\tau, \cdot) dx = \lim_{\delta \rightarrow 0} \frac{1}{2\delta} \int_{\tau - \delta}^{\tau + \delta} \int_{\Omega}\{E\} dxdt.
	\end{equation*}
	
	Equations \eqref{eqcont}, \eqref{momentum equation with measures}, and \eqref{energy inequality with measure}
	form a suitable platform for introducing the measure--valued solutions of the Euler system. To state the exact definition, 
	we make a short excursion in the theory of Young measures.
	
	\subsection{Young measures}
	
	We will introduce some useful notations.
	\begin{definition}
		Let $Q \subseteq \mathbb{R}^d$ be an open set. The mapping $\nu: Q \rightarrow \mathcal{M}(\mathbb{R}^m)$ is said to be \textit{weak-$*$ measurable}, if for all $F \in L^1(Q;C_0(\mathbb{R}^m))$ the function
		\begin{equation*}
			Q \ni x \mapsto \langle \nu_x, F(x,\cdot) \rangle =\int_{\mathbb{R}^m} F(x,\lambda) d\nu_x(\lambda),
		\end{equation*}
		is measurable; here and in the sequel we use the standard notation $\nu_x=\nu(x)$, as if measures $\nu_x$ were parametrized by $x$. In this case, since
		\begin{equation*}
			\|\nu_x\|_{\mathcal{M}(\mathbb{R}^m)}= \sup_{\substack{f \in C_0(\mathbb{R}^m) \\ \|f\|_{\infty}\leq 1}} |\langle \nu_x,f \rangle|,
		\end{equation*}
		the function $x \mapsto 			\|\nu_x\|_{\mathcal{M}(\mathbb{R}^m)}$ is also measurable and we can define
		\begin{equation*}
			\| \nu \|_{L^{\infty}_w(\Omega;\mathcal{M}(\mathbb{R}^m))} = \esssup_{x\in Q} \|\nu_x\|_{\mathcal{M}(\mathbb{R}^m)}.
		\end{equation*}
		Finally, let 
		\begin{equation*}
			L^{\infty}_{weak}(Q;\mathcal{M}(\mathbb{R}^m)) = \{ \nu: Q \rightarrow \mathcal{M}(\mathbb{R}^m); \mbox{ }\nu \mbox{ weak-$*$ measurable, }  	\| \nu \|_{L^{\infty}_w(\Omega;\mathcal{M}(\mathbb{R}^m))} < \infty\}.
		\end{equation*}
	\end{definition}

	Then, the following theorem holds.
	
	\begin{theorem} \label{dual of L1}
		Let $Q\subseteq \mathbb{R}^n$ be open. Let $\Phi \in (L^1(Q;C_0(\mathbb{R}^m)))^*$ be a linear bounded functional. Then there exists a unique $\nu \in L^{\infty}_{weak }(Q; \mathcal{M}(\mathbb{R}^m))$ such that, for all $F \in L^1(Q;C_0(\mathbb{R}^m))$,
		\begin{equation} \label{representation formula}
			\Phi(F)= \int_{Q} \langle \nu_x, F(x) \rangle dx,
		\end{equation}
		and
		\begin{equation*}
		\| \Phi\|_{(L^1(Q;C_0(\mathbb{R}^m)))^*} = \|\nu\|_{L^{\infty}_w(Q; \mathcal{M}(\mathbb{R}^m))}.
		\end{equation*}
	\end{theorem}
	
	\begin{proof}
		See \cite{MalNecRuz}, Chapter 3, Theorem 2.11.
	\end{proof}

	From now on we will consider $Q= (0,T)\times \Omega$ and the sequence $\bm{z}^R=(\varrho_R-\overline{\varrho}, \textbf{m}_R=\varrho_R\textbf{u}_R)$ of solutions to the Navier--Stokes system; we can now construct the \textit{Young measure} associated to the sequence $\{ \bm{z}^R \}_{R>0}$. First, for every $R$ we define the mapping 
	\begin{equation*}
		\nu^R:Q \rightarrow \mathcal{M}(\mathbb{R}^4)
	\end{equation*}
	defined for a.e. $(t,x)\in Q$ by
	\begin{equation*}
		\nu_{t,x}^R =\delta_{\bm{z}^R(t,x)},
	\end{equation*}
	where $\delta_a$ is the Dirac measure supported at $a\in \mathbb{R}^4$. Hence, for every $\psi \in L^1(Q;C_0(\mathbb{R}^4))$ the function
	\begin{equation*}
		(t,x) \mapsto \langle \nu_{t,x}^R, \psi(t,x) \rangle
	\end{equation*}
	is measurable since it is integrable; indeed
	\begin{equation*}
		\langle \nu_{t,x}^R, \psi(t,x)  \rangle =\int_{\mathbb{R}^4} \psi(t,x,\cdot) d\nu_{t,x}^R =\int_{\mathbb{R}^4} \psi(t,x, \cdot) d\delta_{\bm{z}^R(t,x)} = \psi(t,x, \bm{z}^R(t,x)) ,
	\end{equation*}
	and then
	\begin{equation*}
		\int_{0}^{T} \int_{\Omega} |\langle \nu_{t,x}^R, \psi(t,x)  \rangle| dxdt \leq \int_{0}^{T} \int_{\Omega}  \sup_{\lambda \in \mathbb{R}^4} |\psi(t,x,\lambda)| dxdt= \|\psi\|_{L^1(Q;C_0(\mathbb{R}^4))}.
	\end{equation*}
	If we define
	\begin{equation*}
		\nu^R : (t,x)\mapsto \nu_{t,x}^R
	\end{equation*}
	it is weakly-$*$ measurable and we also have that 
	\begin{equation*}
		\|\nu^R \|_{L^{\infty}_w (Q; \mathcal{M}(\mathbb{R}^4))} = \esssup_{(t,x)\in Q} \|\nu_{t,x}^R\|_{\mathcal{M}(\mathbb{R}^4)}= \|\delta_{\bm{z}^R(t,x)}\|_{\mathcal{M}(\mathbb{R}^4)}=1.
	\end{equation*}
	Therefore, $\{\nu^R\}_{R>0}$ is uniformly bounded in $L_{weak}^{\infty}(Q; \mathcal{M}(\mathbb{R}^4))$, which by \hyperref[dual of L1]{Theorem \ref*{dual of L1}} is the dual space of the separable space $L^1(Q;C_0(\mathbb{R}^4))$; we can apply the Banach-Alaoglu theorem to find a subsequence, not relabeled, and $\nu\in L^{\infty}_{weak}(Q; \mathcal{M}(\mathbb{R}^4))$ such that
	\begin{equation*}
		\nu^{R} \overset{*}{\rightharpoonup} \nu \quad \mbox{ in }L^{\infty}_{weak}(Q; \mathcal{M}(\mathbb{R}^4)).
	\end{equation*} 
	This means that for all $\psi \in L^1(Q; C_0(\mathbb{R}^4))$
	\begin{equation*}
		\int_{0}^{T} \int_{\Omega} \psi(t,x, \bm{z}^R(t,x)) dxdt= \int_{0}^{T} \int_{\Omega} \langle \nu_{t,x}^R, \psi(t,x)  \rangle dxdt \rightarrow \int_{0}^{T} \int_{\Omega} \langle \nu_{t,x}, \psi(t,x)  \rangle dxdt \quad \mbox{as } R \rightarrow \infty.
	\end{equation*}
	If we now choose $\psi(t,x, \lambda)= g(t,x)\varphi(\lambda)$ with $g\in L^1(Q)$, $\varphi \in C_0(\mathbb{R}^4)$, the last limit tells us that 
	\begin{equation*}
		\int_{0}^{T} \int_{\Omega} g(t,x) \varphi(\bm{z}^R(t,x)) dxdt= \int_{0}^{T} \int_{\Omega} g(t,x)\langle \nu_{t,x}^R, \varphi  \rangle dxdt \rightarrow \int_{0}^{T} \int_{\Omega} g(t,x)\langle \nu_{t,x}, \varphi  \rangle dxdt \quad \mbox{as } R \rightarrow \infty.
	\end{equation*}
	Then, for every $\varphi \in C_0(\mathbb{R}^4)$, knowing that
	\begin{equation*}
		\varphi(\bm{z}^R) \overset{*}{\rightharpoonup} \overline{\varphi} \quad \mbox{in } L^{\infty}(Q),
	\end{equation*}
	we can deduce that 
	\begin{equation*}
		\overline{\varphi}(t,x) = \langle \nu_{t,x}, \varphi \rangle \quad \mbox{for a.e. } (t,x) \in Q.
	\end{equation*}
	From the weak-$*$ lower semi-continuity of the norm we also have that
	\begin{equation*}
		\|\nu_{t,x}\|_{\mathcal{M}(\mathbb{R}^4)} \leq \liminf_{R \rightarrow \infty} \|\nu_{t,x}^R\|_{\mathcal{M}(\mathbb{R}^4)} =1 \quad \mbox{for a.e. } (t,x)\in Q.
	\end{equation*}
	
	What we proved is the first statement in the following theorem.
	\begin{theorem}
		Let $Q \subset \mathbb{R}^d$ be a measurable set and let $\bm{z}^R:Q \rightarrow \mathbb{R}^s$, $R >0$, be a sequence of measurable functions. Then there exists a subsequence, still denoted by $\bm{z}^R$, and a measure-valued function $\nu$ with the following properties:
		\begin{enumerate}
			\item $\nu \in L^{\infty}_{weak}(Q; \mathcal{M}(\mathbb{R}^s))$, $\|\nu_y\|_{\mathcal{M}(\mathbb{R}^s)}\leq 1$, for a.e. $y\in Q$ and we have for every $\varphi \in C_0(\mathbb{R}^s)$, as $R \rightarrow \infty$,
			\begin{equation*}
				\varphi(\bm{z}^R) \overset{*}{\rightharpoonup} \overline{\varphi} \mbox{ in } L^{\infty}(Q), \quad \overline{\varphi}(y) = \langle \nu_y, \varphi \rangle, \mbox{ for a.e. } y\in Q;
			\end{equation*}
			\item moreover, if
			\begin{equation} \label{condition for a probability measure}
				\lim_{k\rightarrow \infty} \sup_{R>0} \meas \{y\in Q \cap B_r; |\bm{z}^R(y)| \geq k \} =0
			\end{equation}
			for every $r>0$, where $B_r\equiv \{ y\in Q; |y|\leq r \}$, then
			\begin{equation*}
				\|\nu_y\|_{\mathcal{M}(\mathbb{R}^s)}=1 \mbox{ for a.e. }y\in Q;
			\end{equation*}
			\item Let $\Psi: [0,\infty) \rightarrow \mathbb{R}$ be a Young function satisfying the $\Delta_2$-condition. If condition \eqref{condition for a probability measure} holds and if we have for some continuous function $\tau: \mathbb{R}^s \rightarrow \mathbb{R}$
			\begin{equation} \label{condition young meaure coincides with the limit}
				\sup_{R>0} \int_{Q} \Psi(|\tau(\bm{z}^R)|) dy < \infty,
			\end{equation}
			then 
			\begin{equation*}
			\tau(\bm{z}^R) \overset{*}{\rightharpoonup} \overline{\tau} \mbox{ in the Orlicz space } L_{\Psi}(Q) , \quad \overline{\tau}(y) = \langle \nu_y, \tau \rangle \mbox{ for a.e. }y\in Q.
			\end{equation*}
		\end{enumerate}
	\end{theorem}

	\begin{proof}
		See \cite{MalNecRuz}, Chapter 4, Theorem 2.1.
	\end{proof}

	\begin{remark}
		If $\bm{z}^R$ are uniformly bounded in $L^p(Q;\mathbb{R}^s)$ for some $p\in[1,\infty)$, the condition \eqref{condition for a probability measure} is satisfied. Indeed, denoting $A^R_k \equiv \{y\in Q \cap B_r; |\bm{z}^R(y)| \geq k \}$, we have
		\begin{equation*}
			|A^R_k| k^p \leq \int_{A^R_k}|\bm{z}^R(y)|^p dy \leq \int_{Q} |\bm{z}^R(y)|^p dy \leq c.
		\end{equation*}
		Since $c$ is independent of both $R$ and $k$, we obtain
		\begin{equation*}
			\sup_{R >0} |A^R_k| \leq \frac{c}{k^p},
		\end{equation*}
		which implies \eqref{condition for a probability measure}.
	\end{remark}
	
	First, notice that condition \eqref{condition for a probability measure} is satisfied for  $\bm{z}^R=(\varrho_R-\overline{\varrho}, \textbf{m}_R)$; indeed, denoting again $A^R_k \equiv \{y\in Q \cap B_r; |\bm{z}^R(y)| \geq k \}$ we have, for $y\in A_k^R$
	\begin{align*}
		k\leq |(\varrho_R-\overline{\varrho}, \textbf{m}_R)(y)| &\leq |\varrho_R(y)|+ |\textbf{m}_R(y)| \\
		& \leq |[\varrho_R-\overline{\varrho}]_{ess}(y)| + |[\varrho_R-\overline{\varrho}]_{res}(y)|+ |[\textbf{m}_R]_{ess}(y)| + |[\textbf{m}_R]_{res}(y)|,
	\end{align*}
	and hence at least one of the terms on the last line must be $\geq \frac{k}{4}$ so that
	\begin{align*}
		A_k^R &\subseteq \underbrace{ \left\{y\in Q \cap B_r; |[\varrho_R-\overline{\varrho}]_{ess}(y)| \geq \frac{k}{4} \right\}}_{\equiv A_{k,1}^R} \cup \underbrace{ \left\{y\in Q \cap B_r; |[\varrho_R-\overline{\varrho}]_{res}(y)| \geq \frac{k}{4} \right\}}_{\equiv A_{k,2}^R} \\
		& \cup \underbrace{ \left\{y\in Q \cap B_r; |[\textbf{m}_R]_{ess}(y)| \geq \frac{k}{4} \right\}}_{\equiv A_{k,3}^R} \cup \underbrace{ \left\{y\in Q \cap B_r; |[\textbf{m}_R]_{res}(y)| \geq \frac{k}{4} \right\}}_{\equiv A_{k,4}^R}.
	\end{align*}
	For $k$ large enough $(k\geq 4)$, we have
	\begin{align*}
		|A_k^R| k &\leq 4 \sum_{i=1}^{4} |A_{k,i}^R| \frac{k}{4} \\
		& \lesssim |A_{k,1}^R| \left(\frac{k}{4}\right)^2 + |A_{k,2}^R| \left(\frac{k}{4}\right)^{\gamma} + |A_{k,3}^R| \left(\frac{k}{4}\right)^2 + |A_{k,4}^R| \left(\frac{k}{4}\right)^{\frac{2\gamma}{\gamma+1}} \\
		& \leq \int_{A^R_{k,1}} |[\varrho_R-\overline{\varrho}]_{ess}(y)|^2 dy + \int_{A^R_{k,2}} |[\varrho_R-\overline{\varrho}]_{res}(y)|^{\gamma} dy + \int_{A^R_{k,3}} |[\textbf{m}_R]_{ess}(y)|^2 dy + \int_{A^R_{k,4}} |[\textbf{m}_R]_{res}(y)|^{\frac{2\gamma}{\gamma+1}} dy \\
		& \leq \| [\varrho_R-\overline{\varrho}]_{ess} \|_{L^2(Q)} + \| [\varrho_R-\overline{\varrho}]_{res} \|_{L^{\gamma}(Q)} + \| [\textbf{m}_R]_{ess} \|_{L^2(Q)} +\| [\textbf{m}_R]_{res} \|_{L^{\frac{2\gamma}{\gamma+1}}(Q)} \\
		&\leq c(E_0),
	\end{align*}
	where in particular the constant $c(E_0)$ is independent of $k$ and $R$ so that
	\begin{equation*}
		\sup_{R >0} |A_k^R| k \leq \frac{c}{k},
	\end{equation*} 
	which implies \eqref{condition for a probability measure}. Then we obtain that the Young measure in our case is a parametrized family of probability measures supported on the set $[0,\infty)\times \mathbb{R}^N$, since the densities are supposed to be non-negative:
	\begin{equation*}
		\nu_{t,x}:(t,x)\in (0,T) \times \Omega \rightarrow \mathcal{P}([0,\infty)\times \mathbb{R}^N),
	\end{equation*}
	\begin{equation*}
		\nu \in L^{\infty}_{weak} ((0,T)\times \Omega; \mathcal{P}([0,\infty)\times \mathbb{R}^N)).
	\end{equation*} 
	It is also easy to check that $\Psi(t)=t^p$ with $p>1$ are Young functions that satisfy the $\Delta_2$-condition with the constant $2^p$, and in that case $L_{\Psi}(Q)=L^p(Q)$. Thus,
	\begin{enumerate}
		\item first, we can take $\Psi(t)=t^2$ and $\tau_1(\bm{z})=z_1\chi(z_1+\overline{\varrho})$, where $\bm{z}=(z_1,z_2,z_3,z_4)$ in our case, to notice that condition \eqref{condition young meaure coincides with the limit} is equivalent in requiring that $[\varrho_R-\overline{\varrho}]_{ess}$ are uniformly bounded in $L^2((0,T)\times \Omega)$ which is true from \eqref{density bound in L2}. 
		Then we obtain
		\begin{equation*}
			\langle \nu_{t,x}; \tau_1\rangle = f_{\varrho_R-\overline{\varrho}}(t,x) \mbox{ for a.e. } (t,x)\in (0,T)\times \Omega;
		\end{equation*}
		also, taking $\Psi(t)=t^{\gamma}$ and $\tau_2(\bm{z})=z_1(1-\chi(z_1+\overline{\varrho}))$, condition \eqref{condition young meaure coincides with the limit} is equivalent in requiring that $[\varrho_R-\overline{\varrho}]_{res}$ are uniformly bounded in $L^{\gamma}((0,T)\times \Omega)$ which is true from \eqref{density bound in Lgamma}. Then we obtain
		\begin{equation*}
			\langle \nu_{t,x}; \tau_2\rangle = g_{\varrho_R-\overline{\varrho}}(t,x) \mbox{ for a.e. } (t,x)\in (0,T)\times \Omega.
		\end{equation*}
		Unifying the two results we get 
		\begin{equation*}
			\langle \nu_{t,x}; \tau_1+\tau_2\rangle = (\varrho-\overline{\varrho})(t,x) \mbox{ for a.e. } (t,x)\in (0,T)\times \Omega.
		\end{equation*}
		We will write $\langle \nu_{t,x}; \varrho-\overline{\varrho} \rangle= (\varrho-\overline{\varrho})(t,x)$ for almost every $(t,x) \in (0,T)\times \Omega$ just to make the notation readable;
		\item secondly, we can take $\Psi(t)=t^2$ and $\tau_1(\bm{z})=z_i\chi(z_1+\overline{\varrho})$ with $i=2,3,4$ to see that condition \eqref{condition young meaure coincides with the limit} is equivalent in requiring that each component of $[\textbf{m}_R]_{ess}$ is uniformly bounded in $L^2((0,T)\times \Omega)$ which is true from \eqref{momentum bound in L2}. Also, choosing $\Psi(t)=t^{\frac{2\gamma}{\gamma+1}}$ and $\tau_2(\bm{z})=z_i(1-\chi(z_1+\overline{\varrho}))$ with $i=2,3,4$, condition \eqref{condition young meaure coincides with the limit} is equivalent in requiring that each component of $[\textbf{m}_R]_{res}$ is uniformly bounded in $L^{\frac{2\gamma}{\gamma+1}}((0,T)\times \Omega)$ which is true from \eqref{momentum bound in Lq}. Then we obtain
		\begin{equation*}
			\langle \nu_{t,x}; \tau_1+\tau_2\rangle = m_i(t,x) \mbox{ for a.e. } (t,x)\in (0,T)\times \Omega,
		\end{equation*}
		which we will write $\langle \nu_{t,x}; \textbf{m} \rangle = \textbf{m}(t,x)$ for almost every $(t,x) \in (0,T)\times \Omega$.
	\end{enumerate}

	\subsection{Concentration measures and dissipation defect}
	In the previous subsection we showed that the Young measure, applied to proper continuous functions, coincides almost everywhere with the density $\varrho-\overline{\varrho}$ and the momentum $\textbf{m}$. Now, we examine what happens for those functions $H$ for which we only know that 
	\begin{equation*}
		\| H(\varrho_R-\overline{\varrho}, \textbf{m}_R)\|_{L^1((0,T)\times \Omega)} \leq c \quad \mbox{uniformly in } R.
	\end{equation*}
	Without loss of generality, we can consider $|H|$ or, equivalently, assume that $H \geq 0$. We take a family of cut-off functions 
	\begin{equation*}
		T_k(z)= \min\{z,k\};
	\end{equation*}
	Then $T_k(H) \in C_0(\mathbb{R}^4)$ and from the previous construction we know that
	\begin{equation*}
		T_k(H(\varrho_R-\overline{\varrho},\textbf{m}_R)) \overset{*}{\rightharpoonup} \overline{T_k(H)} \quad \mbox{in } L^{\infty}((0,T)\times \Omega)
	\end{equation*}
	with
	\begin{equation*}
		\overline{T_k(H)}(t,x) = \langle \nu_{t,x}, T_k(H) \rangle \quad \mbox{for a.e. } (t,x) \in (0,T) \times \Omega.
	\end{equation*}
	On the other hand we have that 
	\begin{equation*}
		T_k(H)(\lambda) \nearrow H(\lambda), \mbox{ for any }\lambda \in \mathbb{R}^4 \mbox{ as } k\rightarrow \infty,
	\end{equation*}
	thus, by monotone convergence theorem, we have that
	\begin{equation*}
		\langle \nu_{t,x}, T_k(H) \rangle = \int_{\mathbb{R}^4} T_k(H)(\lambda) d\nu_{t,x}(\lambda) \rightarrow \int_{\mathbb{R}^4} H(\lambda) d\nu_{t,x}(\lambda) \quad \mbox{for a.e. } (t,x) \in (0,T) \mbox{ as } k\rightarrow \infty,
	\end{equation*}
	hence $H$ is $\nu_{t,x}$-integrable but the integral can also be infinite. However, by the weak-$*$ lower semi-continuity of the norm  
	\begin{equation*}
		\|\langle \nu_{(\cdot, \cdot)}, T_k(H) \|_{L^1((0,T)\times \Omega)} \leq \liminf_{R \rightarrow \infty} \|T_k(H(\varrho_R-\overline{\varrho}, \textbf{m}_R))\|_{L^1((0,T)\times \Omega)} \leq \liminf_{R\rightarrow \infty} \|H(\varrho_R-\overline{\varrho}, \textbf{m}_R)\|_{L^1((0,T)\times \Omega)} \leq c,
	\end{equation*}
	uniformly in $k$. Then, since
	\begin{itemize}
		\item[(i)] $\lim_{k\rightarrow \infty} \langle \nu_{t,x}; T_k(H) \rangle = \langle \nu_{t,x}; H \rangle$ for a.e. $(t,x)\in (0,T)\times \Omega$;
		\item[(ii)] $\sup_{k\in \mathbb{N}}\| \langle\nu_{(\cdot, \cdot)}; T_k(H)\rangle\|_{L^1((0,T)\times \Omega)} \leq c$,
	\end{itemize}
	applying Fatou's lemma we get that $\| \langle\nu_{(\cdot, \cdot)}; H\rangle\|_{L^1((0,T)\times \Omega)} \leq c$. Then $\langle \nu_{t,x}; H \rangle $ is finite for a.e. $(t,x)\in (0,T)\times \Omega$.
	
	In view of this we can introduce new measures
	\begin{align*}
		&\mu_{\mathbb{M}_{\infty}}= \mu_{\{ \mathbb{M} \}} -\left\langle \nu_{(\cdot, \cdot)}; \frac{\textbf{m}\otimes \textbf{m}}{\varrho} \right\rangle dxdt  ,\\
		&\mu_{p_{\infty}}= \mu_{\{ p \}-p(\overline{\varrho})} -\langle \nu_{(\cdot, \cdot)};p(\varrho)-p(\overline{\varrho})\rangle dxdt, \\
		&\mu_{\sigma_{\infty}} =\mu_{\trace \{\mathbb{M}\}} - \left\langle \nu_{(\cdot, \cdot)}; \frac{|\textbf{m}|^2}{\varrho}\right\rangle dxdt, \\  	&\mu_{E_{\infty}} = \mu_{\{E\}} - \left\langle \nu_{(\cdot, \cdot)}; \frac{1}{2} \frac{|\textbf{m}|^2}{\varrho} +P(\varrho)-P'(\overline{\varrho})(\varrho-\overline{\varrho})-P(\overline{\varrho}) \right \rangle dxdt.
	\end{align*}
	
	Now, revisiting the momentum equation \eqref{momentum equation with measures} and the fact that
	\begin{equation*}
		\divv_x \bm{\varphi} = \mathbb{I}: \nabla_x \bm{\varphi},
	\end{equation*}
	we get
	\begin{align*}
		\int_{\Omega} \textbf{m} \cdot \bm{\varphi}(\tau, \cdot) dx &- \int_{\Omega} \textbf{m}_0 \cdot \bm{\varphi}(0,\cdot) dx \\
		&= \int_{0}^{\tau} \int_{\Omega} \left[ \textbf{m} \cdot \partial_t \bm{\varphi} + \left[\left\langle \nu_{t,x}; \frac{\textbf{m}\otimes \textbf{m}}{\varrho} \right\rangle + \mathbb{M}_{\infty}\right]: \nabla_x \bm{\varphi} + [\langle \nu_{t,x};p(\varrho)\rangle + p_{\infty}] \mathbb{I}: \nabla_x \bm{\varphi} - a\textbf{m}\cdot \bm{\varphi}\right] dx dt,
	\end{align*}
	for all $\tau \in [0,T)$ and for all $\bm{\varphi} \in C^1_c([0,T]\times \overline{\Omega};\mathbb{R}^N)$, $\bm{\varphi}\cdot \textbf{n}|_{\partial \Omega}=0$, which can be rewritten as
	\begin{align*}
		\int_{\Omega} \textbf{m} \cdot \bm{\varphi}(\tau, \cdot) dx - \int_{\Omega} \textbf{m}_0 \cdot \bm{\varphi}(0,\cdot) dx &= \int_{0}^{\tau} \int_{\Omega} \left[ \textbf{m} \cdot \partial_t \bm{\varphi} + \left\langle \nu_{t,x}; \frac{\textbf{m}\otimes \textbf{m}}{\varrho} \right\rangle: \nabla_x \bm{\varphi} + \langle \nu_{t,x};p(\varrho)\rangle \divv_x \bm{\varphi} - a\textbf{m}\cdot \bm{\varphi}\right] dx dt \\
		&+ \int_{0}^{\tau} \int_{\Omega}  \nabla_x \bm{\varphi}: d\mu_m,
	\end{align*}
	for all $\tau \in [0,T)$ and for all $\bm{\varphi} \in C^1_c([0,T]\times \overline{\Omega};\mathbb{R}^N)$, $\bm{\varphi}\cdot \textbf{n}|_{\partial \Omega}=0$, where $ \mu_m=\mathbb{M}_{\infty}+ p_{\infty}\mathbb{I}\in \mathcal{M}([0,T]\times \overline{\Omega}; \mathbb{R}^N\times \mathbb{R}^N)$ is a tensor-valued measure.
	
	Similarly, from \eqref{energy inequality with measure} we get
	\begin{align*}
		\int_{\Omega} &\left[ \left\langle \nu_{\tau, x}; \frac{1}{2} \frac{|\textbf{m}|^2}{\varrho} +P(\varrho)-P'(\overline{\varrho})(\varrho-\overline{\varrho})-P(\overline{\varrho}) \right \rangle + E_{\infty}(\tau)\right] dx + a\int_{0}^{\tau} \int_{\Omega} \left[\left\langle \nu_{t,x}; \frac{|\textbf{m}|^2}{\varrho}\right\rangle + \sigma_{\infty} \right] dxdt \\
		&\quad \quad \quad \quad \quad \quad \quad \quad \quad \quad \quad \quad \leq \int_{\Omega} \left[ \left\langle \nu_{0, x}; \frac{1}{2} \frac{|\textbf{m}|^2}{\varrho} +P(\varrho)-P'(\overline{\varrho})(\varrho-\overline{\varrho})(\varrho-\overline{\varrho})-P(\overline{\varrho}) \right \rangle + E_{\infty}(0) \right]dx,
	\end{align*}
	for a.e. $\tau \in (0,T)$, which can be rewritten as
	\begin{align*}
		\int_{\Omega} \left\langle \nu_{\tau, x}; \frac{1}{2} \frac{|\textbf{m}|^2}{\varrho} +P(\varrho)-P'(\overline{\varrho})(\varrho-\overline{\varrho})-P(\overline{\varrho}) \right \rangle dx &+ a\int_{0}^{\tau} \int_{\Omega} \left\langle \nu_{t,x}; \frac{|\textbf{m}|^2}{\varrho}\right\rangle dxdt +\mathcal{D}(\tau) \\
		&\leq \int_{\Omega}  \left\langle \nu_{0, x}; \frac{1}{2} \frac{|\textbf{m}|^2}{\varrho} +P(\varrho)-P'(\overline{\varrho})(\varrho-\overline{\varrho})-P(\overline{\varrho}) \right \rangle  dx,
	\end{align*}
	for a.e. $\tau \in (0,T)$, with $\mathcal{D}\in L^{\infty}(0,T)$ such that
	\begin{equation*}
		\mathcal{D}(\tau) = \int_{\Omega} E_{\infty}(\tau) dx + a \int_{0}^{\tau} \int_{\Omega} \sigma_{\infty} dxdt
	\end{equation*}
	
	We also have that
	\begin{equation} \label{DDD}
		\int_{0}^{\tau} \int_{\Omega} d|\mu_m| \lesssim \int_{0}^{\tau} \mathcal{D}(t)dt,
	\end{equation}
	for a.e. $\tau \in (0,T)$. Indeed,
	\begin{align*}
		\int_{0}^{\tau} \int_{\Omega} d|\mu_m| &\leq \sum_{i,j=1}^{N}\int_{0}^{\tau} \int_{\Omega} |(\mathbb{M}_{\infty})_{i,j}| dxdt + \sum_{i,j=1}^{N}\int_{0}^{\tau} \int_{\Omega} |p_{\infty}| \delta_{i,j} dxdt \\
		&= \sum_{i,j=1}^{N}\int_{0}^{\tau} \int_{\Omega} |(\mathbb{M}_{\infty})_{i,j}| dxdt + N\int_{0}^{\tau} \int_{\Omega} |p_{\infty}|  dxdt.
	\end{align*}
	Now, we need the following
	\begin{lemma}
		Let $\{\bm{z}^R\}_{R > 0}$, $\bm{z}^R: Q \subset \mathbb{R}^d \rightarrow \mathbb{R}^m$ be a sequence generating a Young measure $\{\nu_y\}_{y\in Q}$, where $Q$ is a measurable set in $\mathbb{R}^d$. Let 
		\begin{equation*}
			G: \mathbb{R}^m \rightarrow [0,\infty)
		\end{equation*}
		be a continuous function such that
		\begin{equation*}
			\sup_{R >0} \|G(\bm{z}^R)\|_{L^1(Q)} < \infty,
		\end{equation*}
		and let $F$ be continuous such that
		\begin{equation*}
			F: \mathbb{R}^m \rightarrow \mathbb{R}, \quad |F(\bm{z})|\leq G(\bm{z}) \mbox{ for all } \bm{z}\in \mathbb{R}^m.
		\end{equation*}
		Denote
		\begin{equation*}
			\mu_{F_{\infty}} = \mu_{\tilde{F}}- \langle \nu_y, F(\textbf{v}) \rangle dy, \quad 
			\mu_{G_{\infty}} = \mu_{\tilde{G}}- \langle \nu_y, G(\textbf{v}) \rangle dy,
		\end{equation*}
		where $ \mu_{\tilde{F}}, \mu_{\tilde{G}} \in \mathcal{M}(Q)$ are the weak-$*$ limits of $\{F(\bm{z}^R)\}_{R > 0}$, $\{G(\bm{z}^R)\}_{R > 0}$ in $\mathcal{M}(Q)$. Then
		\begin{equation*}
			|F_{\infty}| \leq G_{\infty}.
		\end{equation*}
	\end{lemma}
	\begin{proof}
		We have seen that the Young measure $\{\nu_y\}_{y\in Q}$ is such that
		for all $\psi \in L^1(Q; C_0(\mathbb{R}^m))$ 
		\begin{equation*} 
			\int_{Q} \psi(y, \bm{z}^R(y)) dy \rightarrow  \int_{Q} \langle \nu_y, \psi(y)\rangle dy = \int_{Q} \int_{\mathbb{R}^m} \psi(y, \lambda) d\nu_y(\lambda) dy,
		\end{equation*}
		as $R \rightarrow \infty$. Now, from the fact that
		\begin{align*}
			& \mu_{F(\bm{z}^R)} \overset{\ast}{\rightharpoonup} \mu_{\tilde{F}} \quad \mbox{in }\mathcal{M}(Q), \\
			& \mu_{G(\bm{z}^R)} \overset{\ast}{\rightharpoonup} \mu_{\tilde{G}} \quad \mbox{in }\mathcal{M}(Q),
		\end{align*}
		we have that for all $\varphi \in C_0(Q)$
		\begin{align*}
			&\langle \mu_{\tilde{F}}, \varphi \rangle = \lim_{R \rightarrow \infty} \int_{Q} F(\bm{z}^R) \varphi dy = \lim_{R \rightarrow \infty} \int_{\{|\bm{z}^R|\leq M\}} F(\bm{z}^R) \varphi dy + \lim_{R \rightarrow \infty} \int_{\{|\bm{z}^R|> M\}} F(\bm{z}^R) \varphi dy, \\
			& \langle \mu_{\tilde{G}}, \varphi \rangle = \lim_{R \rightarrow \infty} \int_{Q} G(\bm{z}^R) \varphi dy = \lim_{R \rightarrow \infty} \int_{\{|\bm{z}^R|\leq M\}} G(\bm{z}^R) \varphi dy + \lim_{R \rightarrow \infty} \int_{\{|\bm{z}^R|> M\}} G(\bm{z}^R) \varphi dy.
		\end{align*}
		Now, we can write
		\begin{equation*}
			\int_{\{|\bm{z}^R|\leq M\}} F(\bm{z}^R(y)) \varphi(y) dy = \int_{Q} \psi(y,\bm{z}^R(y)) dy,
		\end{equation*}
		with
		\begin{equation*}
			\psi(y,\lambda)= F(\lambda) \varphi(y)\chi_{\{|\lambda|\leq M\}};
		\end{equation*}
		then, we have that $\psi \in L^1(Q; C_0(\mathbb{R}^m))$; indeed, calling $K=\supp (\varphi)$ we have
		\begin{equation*}
			\int_{Q} \|\psi(y, \cdot)\|_{C_0(\mathbb{R}^m)} dy = \int_{K} |\varphi(y)| \sup_{|\lambda|\leq M} |F(\lambda)| \leq |K|\sup_{y\in K} |\varphi(y)| \sup_{|\lambda|\leq M} |F(\lambda)| \leq c,
		\end{equation*}
		since both $\varphi$ and $F$ are continuous functions and so they admit maximum on compact sets. Then, for what we have told previously, we have 
		\begin{align*}
			\lim_{R \rightarrow \infty} \int_{\{|\bm{z}^R|\leq M\}} F(\bm{z}^R) \varphi dy = \lim_{R \rightarrow \infty} \int_{Q} \psi(y,\bm{z}^R(y)) dy = \int_{Q} \langle \nu_y, \psi(y)\rangle dy &= \int_{Q} \int_{\mathbb{R}^m} \psi(y, \lambda) d\nu_y(\lambda) dy \\
			&= \int_{Q} \int_{\{|\lambda|\leq M\}} F(\lambda)\varphi(y) d\nu_y(\lambda) dy.
		\end{align*}
		Applying now Lebesgue theorem we have
		\begin{equation*}
			\lim_{M\rightarrow \infty} \left(\lim_{R \rightarrow \infty} \int_{\{|\bm{z}^R|\leq M\}} F(\bm{z}^R) \varphi dy\right) = \int_{Q} \left(\int_{\mathbb{R}^m} F(\lambda) d\nu_y(\lambda) \right)\varphi dy =  \int_{Q} \langle \nu_y; F \rangle \varphi dy.
		\end{equation*}
		Similarly
		\begin{equation*}
			\lim_{M\rightarrow \infty} \left(\lim_{R \rightarrow \infty} \int_{\{|\bm{z}^R|\leq M\}} G(\bm{z}^R) \varphi dy\right) = \int_{Q} \langle \nu_y; G\rangle \varphi dy. 
		\end{equation*}
		Then, we deduce
		\begin{align*}
			&\langle \mu_{F_{\infty}}, \varphi \rangle = \lim_{M\rightarrow \infty} \left(\lim_{R \rightarrow \infty} \int_{\{|\bm{z}^R|> M\}} F(\bm{z}^R) \varphi dy\right), \\
			&\langle \mu_{G_{\infty}}, \varphi \rangle = \lim_{M\rightarrow \infty} \left(\lim_{R \rightarrow \infty} \int_{\{|\bm{z}^R|> M\}} G(\bm{z}^R) \varphi dy\right).
		\end{align*}
		Then, from condition $|F|\leq G$ we obtain what we wanted to prove.
	\end{proof}
	
	We can apply the lemma with
	\begin{itemize}
		\item $Q= (0,T) \times \Omega \subset \mathbb{R}^{N+1}$;
		\item $m=4$;
		\item $\bm{z}^R=(\varrho_R-\overline{\varrho},\textbf{m}_R)$,
	\end{itemize}
	and with $F=p(\varrho)-p(\overline{\varrho})$, $G=P(\varrho)-P'(\overline{\varrho})(\varrho-\overline{\varrho})-P(\overline{\varrho})$ first and $F=\frac{m_im_j}{\varrho}$, $G=\frac{|\textbf{m}|^2}{\varrho}$ then, to get
	\begin{equation*}
		\int_{0}^{\tau} \int_{\Omega} d|\mu_m| \lesssim \int_{0}^{\tau} \int_{\Omega} E_{\infty} dxdt \leq \int_{0}^{\tau} \mathcal{D}(t) dt.
	\end{equation*}
	
	\subsection{Dissipative measure-valued solution for the compressible Euler system with damping}
	Motivated by the previous discussion, we are ready to introduce the concept of \textit{dissipative measure--valued solution} to the compressible Euler system with damping. It can be seen is a generalization of a similar concept introduced by Gwiazda et al. \cite{GSWW}. While the definition in \cite{GSWW} is based on the description of concentrations via the Alibert--Bouchitt\'e defect measures \cite{AliBou}, our approach is motivated by \cite{FGSWW1}, where the mere inequality \eqref{DDD} is required postulating the domination of the concentrations by the energy dissipation defect. This strategy seems to fit better the studies of singular limits on general physical domains performed in the present paper.
	
	\begin{definition}
		A parametrized family of probability measures
		\begin{equation*}
		\nu_{t,x}:(t,x)\in (0,T) \times \Omega \rightarrow \mathcal{P}([0,\infty)\times \mathbb{R}^N),
		\end{equation*}
		\begin{equation*}
		\nu \in L^{\infty}_{weak} ((0,T)\times \Omega; \mathcal{P}([0,\infty)\times \mathbb{R}^N)),
		\end{equation*} 
		is a \textit{dissipative measure-valued solution} of the problem \eqref{continuity equation in function of the momentum}, \eqref{momentum equation in function of the momentum} with the initial condition $\{\nu_{0, x}\}_{x\in \Omega}$ if
		\begin{enumerate}
			\item the integral identity
			\begin{equation} \label{first condition measure-valued solution}
				\int_{\Omega} \langle \nu_{\tau,x}; \varrho \rangle \varphi dx -\int_{\Omega} \langle \nu_{0,x}; \varrho \rangle \varphi dx = \int_{0}^{\tau} \int_{\Omega} [\langle \nu_{t,x}; \varrho \rangle \partial_t \varphi + \langle \nu_{t,x}; \textbf{m}\rangle \cdot \nabla_x \varphi] dx dt + \int_{0}^{\tau}\int_{\Omega} \nabla_x \varphi \cdot d\mu_c
			\end{equation}
			holds for all $\tau \in[0,T)$, and for all $\varphi \in C^1_c([0,T]\times\overline{\Omega})$, where $\mu_c \in \mathcal{M}([0,T]\times \overline{\Omega}; \mathbb{R}^N)$ is a vector--valued measure;
			\item the integral identity
			\begin{equation} \label{second condition measure-valued solution}
			\begin{aligned}
				\int_{\Omega} \langle \nu_{\tau,x};\textbf{m}\rangle &\cdot \bm{\varphi}(\tau, \cdot) dx - \int_{\Omega} \langle \nu_{0,x};\textbf{m}\rangle \cdot \bm{\varphi}(0,\cdot) dx \\
				&= \int_{0}^{\tau} \int_{\Omega} \left[ \langle \nu_{t,x}; \textbf{m}\rangle \cdot \partial_t \bm{\varphi} + \left\langle \nu_{t,x}; \frac{\textbf{m}\otimes \textbf{m}}{\varrho} \right\rangle: \nabla_x \bm{\varphi} + \langle \nu_{t,x};p(\varrho)\rangle \divv_x \bm{\varphi} - a\langle \nu_{t,x}; \textbf{m}\rangle\cdot \bm{\varphi}\right] dx dt \\
				&+ \int_{0}^{\tau} \int_{\Omega}  \nabla_x \bm{\varphi}: d\mu_m,
			\end{aligned}
			\end{equation}
			holds for all $\tau \in [0,T)$ and for all $\bm{\varphi} \in C^1_c([0,T]\times \overline{\Omega};\mathbb{R}^N)$, $\bm{\varphi}\cdot \textbf{n}|_{\partial \Omega}$, where $\mu_m\in \mathcal{M}([0,T]\times \overline{\Omega}; \mathbb{R}^N\times \mathbb{R}^N)$ is a tensor--valued measure; both $\mu_c, \mu_m$ are called \textit{concentration measures};
			\item the following inequality
			\begin{align} \label{third condition measure-valued solution}
			\begin{aligned}
				\int_{\Omega} \left\langle \nu_{\tau, x}; \frac{1}{2} \frac{|\textbf{m}|^2}{\varrho} +P(\varrho)-P'(\overline{\varrho})(\varrho-\overline{\varrho})-P(\overline{\varrho}) \right \rangle dx &+ a\int_{0}^{\tau} \int_{\Omega} \left\langle \nu_{t,x}; \frac{|\textbf{m}|^2}{\varrho}\right\rangle dxdt +\mathcal{D}(\tau)\\
				&\leq \int_{\Omega}  \left\langle \nu_{0, x}; \frac{1}{2} \frac{|\textbf{m}|^2}{\varrho} +P(\varrho)-P'(\overline{\varrho})(\varrho-\overline{\varrho})-P(\overline{\varrho}) \right \rangle  dx,
			\end{aligned}			
			\end{align}
			holds for a.e. $\tau \in (0,T)$, where $\mathcal{D}\in L^{\infty}(0,T)$, $\mathcal{D}\geq 0$ is called \textit{dissipation defect} of the total energy;
			\item there exists a constant $C>0$ such that
			\begin{equation} \label{fourth condition measure-valued solutions}
				\int_{0}^{\tau} \int_{\Omega} d|\mu_c| + \int_{0}^{\tau} \int_{\Omega} d|\mu_m| \leq C \int_{0}^{\tau} \mathcal{D}(t) dt,
			\end{equation}	
			for a.e. $\tau \in (0,T)$.
		\end{enumerate}
	\end{definition}

	Now, summarizing the discussion concerning the vanishing viscosity limit of the Navier--Stokes system, we can state the first result of the present paper.
	
	\begin{theorem} \label{Existence}
		Let $\Omega \subset \mathbb{R}^N$, $N=2,3$ be a domain with compact Lipschitz boundary and $\overline{\varrho} \geq 0$ be a given far field density if $\Omega$ is unbounded. Suppose that $\gamma > \frac{N}{2}$ and let $\varrho^R$, $\textbf{u}^R$ be a family of weak solutions to the Navier--Stokes system \eqref{continuity equation N-S} -- \eqref{boundary condition N-S} in 
		\begin{equation*}
			(0,T) \times \Omega_R, \ \Omega_R = \Omega \cap B_R.
		\end{equation*}
		Let the corresponding initial data $\varrho_0$, $\textbf{u}_0$ be independent of $R$ satisfying
		\begin{equation*}
			\varrho_0 > 0 ,\ \int_{\Omega} \left[\frac{1}{2}\varrho_0|\textbf{u}_0|^2+ P(\varrho_0) - P'(\overline{\varrho}) (\varrho_0 - \overline{\varrho}) - P(\overline{\varrho})\right] dx \leq E_0.
		\end{equation*}
		
		Then the family $\{ \varrho^R , \textbf{m}^R = \varrho^R \textbf{u}^R \}_{R > 0}$ generates, as $R \to \infty$, a Young measure 
		$\{ \nu_{t,x} \}_{t \in (0,T); x \in \Omega}$ which is a dissipative measure--valued solution of the Euler system \eqref{continuity equation in function of the momentum}, \eqref{momentum equation in function of the momentum}.
	\end{theorem}

	\section{Weak-strong uniqueness}
	
	Our next goal is to show that the dissipative measure-valued solutions introduced in the previous section satisfy an extended version of the energy inequality \eqref{third condition measure-valued solution} known as \textit{relative energy inequality}.
	
	We introduce the \textit{relative energy functional}:
	\begin{equation*}
		\mathcal{E}(\nu=\nu_{t,x}(\varrho,\textbf{m})|r,\textbf{U})= \int_{\Omega} \left\langle \nu_{t,x}; \frac{1}{2\varrho}(|\textbf{m}-\varrho \textbf{U}|^2)+P(\varrho)- P'(r)(\varrho-r)-P(r) \right\rangle dx,
	\end{equation*}
	If $\varrho \mapsto p(\varrho)$ is strictly increasing in $(0,\infty)$, which is true in our case, then the pressure potential $P$ is strictly convex; indeed
	\begin{equation*}
		P''(\varrho)= \frac{p'(\varrho)}{\varrho} >0.
	\end{equation*}
	For a differentiable function this is equivalent in saying that the function lies above all of its tangents:
	\begin{equation*}
		P(\varrho) \geq P'(r)(\varrho -r) + P(r)
	\end{equation*}
	for all $\varrho, r \in (0,\infty)$, and the equality holds if and only if $\varrho=r$. Thus, we deduced that $\mathcal{E}\geq 0$, where equality holds if and only if
	\begin{equation*}
		\nu_{t,x}= \delta_{r(t,x), r(t,x)\textbf{U}(t,x)} \quad \mbox{for a.e. }(t,x) \in (0,T) \times \Omega.
	\end{equation*}
	
	We can now prove the following
	
	\begin{theorem}
		Let $[r,\textbf{U}]$ be a strong solution of the compressible Euler system with damping with compactly supported initial data so that $\textbf{U} \in C^{\infty}_c([0,T]\times \overline{\Omega}; \mathbb{R}^N)$, where in particular $\textbf{U}\cdot \textbf{n}|_{\partial \Omega}=0$, and $r-\overline{\varrho}\in C^{\infty}_c([0,T]\times \overline{\Omega})$ with $r>0$. Let $\{ \nu_{t,x} \}_{(t,x)\in (0,T)\times \Omega}$ be a dissipative measure-valued solution of the same system (in terms of $\varrho$ and the momentum $\textbf{m}$), with a dissipation defect $\mathcal{D}$ and such that 
		\begin{equation} \label{same initial data}
			\nu_{0,x}= \delta_{r(0,x),(r\textbf{U})(0,x)} \quad \mbox{for a.e. }x\in \Omega.
		\end{equation}
		Then $\mathcal{D}=0$ and 
		\begin{equation*}
			\nu_{t,x}= \delta_{r(t,x), (r\textbf{U})(t,x)} \quad \mbox{for a.e. }(t,x)\in (0,T)\times \Omega.
		\end{equation*}
	\end{theorem}
	
	\begin{remark}
		Note that we must have $\overline{\varrho} > 0$ if $\Omega$ is unbounded.
	\end{remark}
	\begin{proof}
		It is enough to prove that $\mathcal{E}(\tau)=0$ for all $\tau \in (0,T)$. We can take $\textbf{U}$ as a test function in the momentum equation \eqref{second condition measure-valued solution} to obtain
		\begin{align*}
			\left[\int_{\Omega} \langle \nu_{t,x}; \textbf{m} \rangle \cdot \textbf{U} dx \right]_{t=0}^{t=\tau} &= \int_{0}^{\tau} \int_{\Omega} \left[ \langle \nu_{t,x}; \textbf{m} \rangle \cdot \partial_t \textbf{U}+ \left\langle \nu_{t,x}; \frac{\textbf{m}\otimes \textbf{m}}{\varrho} \right\rangle: \nabla_x \textbf{U}+ \langle\nu_{t,x}; p(\varrho)\rangle \mbox{div}_x \textbf{U} \right] dx dt \\
			&-a\int_{0}^{\tau} \int_{\Omega}\langle \nu_{t,x};\textbf{m}\rangle \cdot \textbf{U} + \int_{0}^{\tau} \int_{\Omega} \nabla_x \textbf{U}: d\mu_m;
		\end{align*}
		and $\frac{1}{2} |\textbf{U}|^2$ as a test function in the continuity equation \eqref{first condition measure-valued solution} to get
		\begin{equation*}
			\left[\frac{1}{2}\int_{\Omega} \langle \nu_{t,x}; \varrho \rangle |\textbf{U}|^2 dx\right]_{t=0}^{t=\tau}= \int_{0}^{\tau} \int_{\Omega} [\langle \nu_{t,x}; \varrho \rangle \textbf{U}\cdot\partial_t \textbf{U} + \langle \nu_{t,x}; \textbf{m}\rangle \cdot \nabla_x \textbf{U} \cdot \textbf{U}] dx dt + \int_{0}^{\tau} \int_{\Omega} \textbf{U} \cdot \nabla_x \textbf{U}\cdot d\mu_c.
		\end{equation*}
		Finally, take $P'(r)-P'(\overline{\varrho})$ as test function in \eqref{first condition measure-valued solution} to get
		\begin{equation*}
			\left[ \int_{\Omega} \langle \nu_{t,x}; \varrho \rangle (P'(r)(t,\cdot)- P'(\overline{\varrho}))dx\right]_{t=0}^{t=\tau} = \int_{0}^{\tau} \int_{\Omega} [\langle \nu_{t,x}; \varrho \rangle \partial_t P'(r) + \langle \nu_{t,x}; \textbf{m}\rangle \cdot \nabla_x P'(r)] dx dt + \int_{0}^{\tau} \int_{\Omega} \nabla_x P'(r)\cdot d\mu_c.
		\end{equation*}
		Then, from the energy inequality \eqref{third condition measure-valued solution}, summing up all these terms we get
		\begin{equation*}
			\left[ \int_{\Omega} \left\langle \nu_{t,x}; \frac{1}{2\varrho} |\textbf{m}-\varrho \textbf{U}|^2+ P(\varrho)-\varrho P'(r)+\overline{\varrho}P(\overline{\varrho}) \right\rangle dx\right]_{t=0}^{t=\tau} +a\int_{0}^{\tau}\int_{\Omega} \left\langle \nu_{t,x}; \frac{\textbf{m}}{\varrho} \cdot (\textbf{m}-\varrho \textbf{U}) \right\rangle dxdt + \mathcal{D}(\tau)
		\end{equation*}
		\begin{align*}
			&\leq \int_{0}^{\tau} \int_{\Omega} \langle \nu_{t,x}; \varrho\textbf{U} -\textbf{m} \rangle \cdot[\partial_t \textbf{U} + \nabla_x \textbf{U} \cdot \textbf{U}] dxdt \\
			&+ \int_{0}^{\tau} \int_{\Omega} \left\langle \nu_{t,x}; \frac{(\textbf{m}-\varrho \textbf{U})\otimes (\varrho \textbf{U}-\textbf{m}) }{\varrho} \right\rangle: \nabla_x \textbf{U} dxdt \\
			& - \int_{0}^{\tau} \int_{\Omega} \langle \nu_{t,x}; p(\varrho)\rangle \divv_x \textbf{U} dxdt -\int_{0}^{\tau}\int_{\Omega} [\langle\nu_{t,x};\varrho\rangle \partial_tP'(r)+ \langle\nu_{t,x};\textbf{m}\rangle \cdot \nabla_x P'(r)] dxdt \\
			&- \int_{0}^{\tau} \int_{\Omega} \nabla_x \textbf{U}: d\mu_m + \int_{0}^{\tau} \int_{\Omega} \textbf{U}\cdot \nabla_x \textbf{U} \cdot d\mu_c -\int_{0}^{\tau} \int_{\Omega} \nabla_x P'(r)\cdot d\mu_c.
		\end{align*}
		Notice that the term 
		\begin{equation*}
			\frac{\textbf{m}}{\varrho} \cdot (\textbf{m}-\varrho \textbf{U}) = \frac{|\textbf{m}|^2}{\varrho}- \textbf{m}\cdot \textbf{U}
		\end{equation*}
		is well-defined and integrable. We have
		\begin{equation*}
			P(\varrho)-\varrho P'(r)+\overline{\varrho}P(\overline{\varrho}) = P(\varrho) -P'(r)(\varrho-r) -P(r) -[rP'(r)-P(r)-\overline{\varrho}P'(\overline{\varrho})],
		\end{equation*}
		where, since $P(\overline{\varrho})=0$,
		\begin{equation*}
			rP'(r)-P(r)-\overline{\varrho}P'(\overline{\varrho}) = p(r)-p(\overline{\varrho}).
		\end{equation*}
		Then
		\begin{equation*}
			\left[\int_{\Omega} \langle \nu_{t,x}; p(r)-p(\overline{\varrho}) \rangle dx\right]_{t=0}^{t=\tau} =\int_{0}^{\tau} \int_{\Omega} \langle \nu_{t,x};\partial_t(p(r)-p(\overline{\varrho})) \rangle dxdt= \int_{0}^{\tau} \int_{\Omega} \langle \nu_{t,x};\partial_t p(r) \rangle dxdt
		\end{equation*}
		and knowing that the pressure potential satisfy the equation 
		\begin{equation} \label{relation for the pressure potential}
			\partial_t P(r) + \divv_x(P(r)\textbf{U}) +p(r)\divv_x\textbf{U}=0,
		\end{equation}
		we can deduce that 
		\begin{equation} \label{relation}
			\int_{\Omega} \langle \nu_{t,x};\partial_t p(r) \rangle dx= \int_{\Omega} \langle \nu_{t,x}; r\partial_t P'(r) + r\nabla_x P'(r)\cdot \textbf{U} + p(r)\divv_x \textbf{U} \rangle dx.
		\end{equation}
		We obtain the \textit{relative energy inequality}:
		\begin{equation} \label{relative nergy inequality}
			\begin{aligned}
				\left[ \mathcal{E}(\nu|r,\textbf{U}) \right]_{t=0}^{t=\tau} &+a\int_{0}^{\tau}\int_{\Omega} \left\langle \nu_{t,x}; \frac{\textbf{m}}{\varrho} \cdot (\textbf{m}-\varrho \textbf{U}) \right\rangle dxdt + \mathcal{D}(\tau) \\
				&\leq \int_{0}^{\tau} \int_{\Omega} \langle \nu_{t,x}; \varrho\textbf{U} -\textbf{m} \rangle \cdot[\partial_t \textbf{U} + \nabla_x \textbf{U} \cdot \textbf{U}] dxdt \\
				&+ \int_{0}^{\tau} \int_{\Omega} \left\langle \nu_{t,x}; 	\frac{(\textbf{m}-\varrho \textbf{U})\otimes (\varrho \textbf{U}-\textbf{m}) }{\varrho} \right\rangle: \nabla_x \textbf{U} dxdt \\
				& - \int_{0}^{\tau} \int_{\Omega} \langle \nu_{t,x}; p(\varrho)-p(r)\rangle \divv_x \textbf{U} dxdt \\
				& -\int_{0}^{\tau}\int_{\Omega} [\langle\nu_{t,x};(\varrho -r)\partial_tP'(r)+ (\textbf{m}-r\textbf{U}) \cdot  \nabla_x P'(r)\rangle dxdt \\
				&- \int_{0}^{\tau} \int_{\Omega} \nabla_x \textbf{U}: 	d\mu_m + \int_{0}^{\tau} \int_{\Omega} \textbf{U}\cdot \nabla_x \textbf{U} \cdot d\mu_c -\int_{0}^{\tau} \int_{\Omega} \nabla_x P'(r)\cdot d\mu_c.
			\end{aligned}
		\end{equation}  
		Now we can use the fact that $[r,\textbf{U}]$ is a strong solution: from the momentum equation we can deduce that
		\begin{equation*}
			\partial_t \textbf{U} + \textbf{U}\cdot \nabla_x \textbf{U} = -\frac{1}{r} \nabla_x p(r) - a \textbf{U}= -P''(r) \nabla_x r -a\textbf{U}=- \nabla_xP'(r) -a\textbf{U};
		\end{equation*}
		substituting, we get
		\begin{align*}
			\left[ \mathcal{E}(\nu|r,\textbf{U}) \right]_{t=0}^{t=\tau} &+a\int_{0}^{\tau}\int_{\Omega} \left\langle \nu_{t,x}; \frac{1}{2\varrho} |\textbf{m}-\varrho\textbf{U}|^2 \right\rangle dxdt + \mathcal{D}(\tau) \\
			&\leq \int_{0}^{\tau} \int_{\Omega} \left\langle \nu_{t,x}; \frac{(\textbf{m}-\varrho \textbf{U})\otimes (\varrho \textbf{U}-\textbf{m}) }{\varrho} \right\rangle: \nabla_x \textbf{U} dxdt \\
			& - \int_{0}^{\tau} \int_{\Omega} \langle \nu_{t,x}; p(\varrho)-p(r)\rangle \divv_x \textbf{U} dxdt \\
			& -\int_{0}^{\tau}\int_{\Omega} [\langle\nu_{t,x};P''(r)(\varrho-r)[\partial_tr+\nabla_x r \cdot \textbf{U} ]\rangle dxdt \\
			&- \int_{0}^{\tau} \int_{\Omega} \nabla_x \textbf{U}: d\mu_m + \int_{0}^{\tau} \int_{\Omega} \textbf{U}\cdot \nabla_x \textbf{U} \cdot d\mu_c -\int_{0}^{\tau} \int_{\Omega} \nabla_x P'(r)\cdot d\mu_c.
		\end{align*}
		From the continuity equation we also have
		\begin{equation*}
			\partial_t r+\nabla_x r \cdot \textbf{U}= -r \divv_x \textbf{U},
		\end{equation*}
		and thus, knowing that $rP''(r)=p'(r)$, we get
		\begin{align*}
			\left[ \mathcal{E}(\nu|r,\textbf{U}) \right]_{t=0}^{t=\tau} &+a\int_{0}^{\tau}\int_{\Omega} \left\langle \nu_{t,x}; \frac{1}{2\varrho} |\textbf{m}-\varrho\textbf{U}|^2 \right\rangle dxdt + \mathcal{D}(\tau) \\
			&\leq \int_{0}^{\tau} \int_{\Omega} \left\langle \nu_{t,x}; \frac{(\textbf{m}-\varrho \textbf{U})\otimes (\varrho \textbf{U}-\textbf{m}) }{\varrho} \right\rangle: \nabla_x \textbf{U} dxdt \\
			& - \int_{0}^{\tau} \int_{\Omega} \langle \nu_{t,x}; 	p(\varrho)-p'(r)(\varrho-r)-p(r)\rangle \divv_x \textbf{U} dxdt \\
			&- \int_{0}^{\tau} \int_{\Omega} \nabla_x \textbf{U}: d\mu_m + \int_{0}^{\tau} \int_{\Omega} \textbf{U}\cdot \nabla_x \textbf{U} \cdot d\mu_c -\int_{0}^{\tau} \int_{\Omega} \nabla_x P'(r)\cdot d\mu_c.
		\end{align*}
		Finally, using the fact that the initial data are the same and thus $\mathcal{E}(\nu|r,\textbf{U})(0)=0$, we end up to
		\begin{align*}
			\mathcal{E}(\nu|r,\textbf{U})(\tau) &+a\int_{0}^{\tau}\int_{\Omega} \left\langle \nu_{t,x}; \frac{1}{2\varrho} |\textbf{m}-\varrho\textbf{U}|^2 \right\rangle dxdt + \mathcal{D}(\tau) \\
			&\leq \int_{0}^{\tau} \int_{\Omega} \left\langle \nu_{t,x}; \left|\frac{(\textbf{m}-\varrho \textbf{U})\otimes (\varrho \textbf{U}-\textbf{m}) }{\varrho}\right| \right\rangle |\nabla_x \textbf{U}| dxdt \\
			& + \int_{0}^{\tau} \int_{\Omega} \langle \nu_{t,x}; |p(\varrho)-p'(r)(\varrho-r)-p(r)|\rangle |\divv_x \textbf{U}| dxdt \\
			&+ \int_{0}^{\tau} \int_{\Omega} |\nabla_x \textbf{U}|\cdot d|\mu_m| + \int_{0}^{\tau} \int_{\Omega} |\textbf{U}\cdot \nabla_x \textbf{U}| \cdot d|\mu_c| + \int_{0}^{\tau} \int_{\Omega} |\nabla_x P'(r)|\cdot d|\mu_c|.
		\end{align*}
		Since $\textbf{U}$ and $P'(r)-P(\overline{\varrho})$ have compact support we can control the terms $|\nabla_x \textbf{U}|$, $|\divv_x \textbf{U}|$, $|\textbf{U}\cdot \nabla_x \textbf{U}|$ and $|\nabla_x P'(r)|$ by some constants. It is also obvious that there exist a constant $c_1$ such that
		\begin{equation*}
			\left| \frac{(\textbf{m}- \varrho \textbf{U}) \otimes (\varrho \textbf{U}-\textbf{m})}{\varrho} \right| \leq \frac{c_1}{2\varrho} |\textbf{m}-\varrho \textbf{U}|^2,
		\end{equation*}
		and a constant $c_2$ such that
		\begin{equation*}
			|p(\varrho) -p'(r)(\varrho - r)- p(r)| \leq c_2 (P(\varrho) -P'(r)(\varrho - r)- P(r)).
		\end{equation*}
		Thus
		\begin{equation*}
			\mathcal{E}(\varrho,\textbf{m}|r,\textbf{U})(\tau) + \mathcal{D}(\tau) \leq c \int_{0}^{\tau} \left[\mathcal{E}(\varrho,\textbf{m}|r,\textbf{U})(t)+\mathcal{D}(t)\right] dt.
		\end{equation*}
		By Gronwall lemma we obtain
		\begin{equation*}
			\mathcal{E}(\varrho,\textbf{m}|r,\textbf{U})(\tau) + \mathcal{D}(\tau)\leq 0 \quad \mbox{for all }\tau \in (0,T).
		\end{equation*}
		But  since $\mathcal{E},\mathcal{D} \geq 0$ this implies $\mathcal{D}(\tau)=0$ and $\mathcal{E}(\tau)=0$ for all $\tau \in (0,T)$.
	\end{proof}

	Notice that the relative energy inequality \eqref{relative nergy inequality} is true for general functions $r-\overline{\varrho} \in C_c^{\infty}([0,T]\times \overline{\Omega})$, $\textbf{U} \in C_c^{\infty}([0,T]\times \overline{\Omega}; \mathbb{R}^N)$, not necessarily strong solutions to the Euler system. Then, using a density argument, we can prove the following result.
	
	\begin{theorem} \label{theorem weak-strong uniqueness}
		Let $[r,\textbf{U}]$ be a strong solution of the compressible Euler system with damping such that $\textbf{U} \in C([0,T]; H^M(\Omega; \mathbb{R}^N))$, $M>\frac{N}{2}+1$, where in particular $\textbf{U}\cdot \textbf{n}|_{\partial \Omega}=0$, and $r-\overline{\varrho}\in C([0,T]; H^M(\Omega))$ with $r>0$. Let $\{ \nu_{t,x} \}_{(t,x)\in (0,T)\times \Omega}$ be a dissipative measure-valued solution of the same system (in terms of $\varrho$ and the momentum $\textbf{m}$), with a dissipation defect $\mathcal{D}$ and such that 
		\begin{equation*} 
			\nu_{0,x}= \delta_{r(0,x),(r\textbf{U})(0,x)} \quad \mbox{for a.e. }x\in \Omega.
		\end{equation*}
		Then $\mathcal{D}=0$ and 
		\begin{equation*}
			\nu_{t,x}= \delta_{r(t,x), (r\textbf{U})(t,x)} \quad \mbox{for a.e. }(t,x)\in (0,T)\times \Omega.
		\end{equation*}
	\end{theorem}

	\begin{proof}
		We will first prove that the relative energy inequality \eqref{relative nergy inequality} holds for $[r, \textbf{U}]$ as in our hypothesis.
		By density, we can find two sequences $\{ r_n-\overline{\varrho} \}_{n\in \mathbb{N}} \subset C_c^{\infty}([0,T]\times \overline{\Omega})$, $\{ \textbf{U}_n \}_{n\in \mathbb{N}} \subset C_c^{\infty}([0,T]\times \overline{\Omega}; \mathbb{R}^N)$ such that 
		\begin{align*}
			r_n-\overline{\varrho} \rightarrow r-\overline{\varrho} \quad &\mbox{in }C([0,T]; H^M(\Omega)), \\
			\textbf{U}_n \rightarrow \textbf{U} \quad &\mbox{in } C([0,T]; H^M(\Omega; \mathbb{R}^N)).
		\end{align*}
		If we now fix $\varepsilon>0$, we know that there exists $n_0=n_0(\varepsilon)$ such that, for every $n\geq n_0$
		\begin{equation*}
			\sup_{t\in [0,T]} \| (r - r_n) (t,\cdot) \|_{H^M(\Omega)} < \varepsilon,
		\end{equation*}
		\begin{equation*}
			\sup_{t\in [0,T]} \| (\textbf{U}- \textbf{U}_n) (t,\cdot) \|_{H^M(\Omega; \mathbb{R}^N)} < \varepsilon.
		\end{equation*}
		From now on, let $n\geq n_0$; for each $t\in [0,T]$ we have
		\begin{align*}
			\int_{\Omega} \left \langle \nu_{t,x}; \frac{1}{2\varrho}|\textbf{m}-\varrho \textbf{U}|^2\right \rangle dx &= \int_{\Omega} \left \langle \nu_{t,x}; \frac{1}{2\varrho}|\textbf{m}-\varrho (\textbf{U}-\textbf{U}_n + \textbf{U}_n)|^2\right \rangle dx \\
			&= \int_{\Omega} \left \langle \nu_{t,x}; \frac{1}{2\varrho}|\textbf{m}-\varrho \textbf{U}_n|^2\right \rangle dx  \\
			&- \int_{\Omega} \langle \nu_{t,x}; \textbf{m}-(\varrho-\overline{\varrho}) \textbf{U}_n \rangle \cdot (\textbf{U}-\textbf{U}_n) (t, \cdot)dx + \overline{\varrho}\int_{\Omega} \textbf{U}_n \cdot (\textbf{U}-\textbf{U}_n) (t, \cdot)dx \\
			&+ \frac{1}{2} \int_{\Omega} \langle \nu_{t, x}; \varrho-\overline{\varrho} \rangle   |\textbf{U}-\textbf{U}_n|^2(t,\cdot) dx + \frac{\overline{\varrho}}{2} \int_{\Omega} |\textbf{U}-\textbf{U}_n|^2 (t,\cdot) dx.
		\end{align*}
		Revoking notation introduced in Section \ref{WSS},
		we focus on the last two lines line: we can rewrite the first term as
		\begin{equation*}
			\int_{\Omega} \langle \nu_{t,x}; [\textbf{m}]_{ess}-[\varrho-\overline{\varrho}]_{ess} \textbf{U}_n \rangle \cdot (\textbf{U}-\textbf{U}_n) (t, \cdot)dx + \int_{\Omega} \langle \nu_{t,x}; [\textbf{m}]_{res}-[\varrho-\overline{\varrho}]_{res} \textbf{U}_n \rangle \cdot (\textbf{U}-\textbf{U}_n) (t, \cdot)dx;
		\end{equation*}
		since $\langle \nu_{(t, \cdot)};[\textbf{m}]_{ess}-[\varrho-\overline{\varrho}]_{ess} \textbf{U}_n \rangle, (\textbf{U}-\textbf{U}_n) (t, \cdot) \in L^2(\Omega; \mathbb{R}^N)$  we can apply H\"{o}lder's inequality to get 
		\begin{align*}
			\int_{\Omega} \langle \nu_{t,x}; [\textbf{m}]_{ess}-[\varrho-\overline{\varrho}]_{ess} \textbf{U}_n \rangle &\cdot (\textbf{U}-\textbf{U}_n) (t, \cdot)dx \\
			&\leq \sup_{t\in [0,T]} \| \langle \nu_{(t, \cdot)}; [\textbf{m}]_{ess}-[\varrho-\overline{\varrho}]_{ess} \textbf{U}_n \rangle\|_{L^2(\Omega; \mathbb{R}^N)} \| (\textbf{U}-\textbf{U}_n) (t, \cdot)\|_{L^2(\Omega; \mathbb{R}^N)} \\
			& \leq C \sup_{t\in [0,T]} \| (\textbf{U}- \textbf{U}_n) (t,\cdot) \|_{H^M(\Omega; \mathbb{R}^N)} \\
			& \leq C\varepsilon.
		\end{align*}
		We also have that $\langle \nu_{(t, \cdot)};[\varrho-\overline{\varrho}]_{res} \textbf{U}_n \rangle \in L^{\gamma}(K); \mathbb{R}^N)$  with $K$ compact and since $\gamma>\frac{2\gamma}{\gamma+1}$ we obtain $\langle \nu_{(t, \cdot)};[\textbf{m}]_{res}-[\varrho-\overline{\varrho}]_{res} \textbf{U}_n \rangle \in L^{\frac{2\gamma}{\gamma+1}}(\Omega; \mathbb{R}^N)$; using the embedding of the Sobolev space into the H\"{o}lder one we get that $(\textbf{U}-\textbf{U}_n) (t, \cdot) \in L^{\infty}(\Omega; \mathbb{R}^N)$ and hence $(\textbf{U}-\textbf{U}_n) (t, \cdot) \in L^p(\Omega; \mathbb{R}^N)$ for all $p\in [2,\infty]$. Since $\frac{2\gamma}{\gamma-1}>2$, we can again apply H\"{o}lder's inequality to get
		\begin{align*}
			\int_{\Omega} \langle \nu_{t,x}; [\textbf{m}]_{res}-[\varrho-\overline{\varrho}]_{res} \textbf{U}_n \rangle &\cdot (\textbf{U}-\textbf{U}_n) (t, \cdot)dx \\
			&\leq \sup_{t\in [0,T]} \| \langle \nu_{(t, \cdot)}; [\textbf{m}]_{res}-[\varrho-\overline{\varrho}]_{res} \textbf{U}_n \rangle\|_{L^{\frac{2\gamma}{\gamma+1}}(\Omega; \mathbb{R}^N)} \| (\textbf{U}-\textbf{U}_n) (t, \cdot)\|_{L^{\frac{2\gamma}{\gamma-1}}(\Omega; \mathbb{R}^N)} \\
			& \leq C \sup_{t\in [0,T]} \| (\textbf{U}- \textbf{U}_n) (t,\cdot) \|_{H^M(\Omega; \mathbb{R}^N)} \\
			& \leq C \varepsilon.
		\end{align*}
		For the second term we can apply H\"{o}lder's inequality:
		\begin{align*}
			\int_{\Omega} \textbf{U}_n \cdot (\textbf{U}-\textbf{U}_n)(t,\cdot) dx & \leq \sup_{t\in [0,T]} \| \textbf{U}_n(t,\cdot) \|_{L^2(\Omega; \mathbb{R}^N)} \| (\textbf{U}-\textbf{U}_n) (t, \cdot)\|_{L^2(\Omega; \mathbb{R}^N)} \\
			& \leq C \sup_{t\in [0,T]} \| (\textbf{U}- \textbf{U}_n) (t,\cdot) \|_{H^M(\Omega; \mathbb{R}^N)} \\
			& \leq C \varepsilon.
		\end{align*}
		Applying the same procedure as before to the third term we get
		\begin{align*}
			\int_{\Omega} \langle \nu_{t, x}; \varrho-\overline{\varrho} \rangle  |\textbf{U}-\textbf{U}_n|^2(t,\cdot) dx &= \int_{\Omega} \langle \nu_{t, x}; [\varrho-\overline{\varrho}]_{ess}+ [\varrho-\overline{\varrho}]_{res} \rangle  |\textbf{U}-\textbf{U}_n|^2(t,\cdot) dx \\
			& \leq \sup_{t\in [0,T]} \| \langle \nu_{(t, \cdot)}; [\varrho-\overline{\varrho}]_{ess} \rangle\|_{L^2(\Omega; \mathbb{R}^N)} \| (\textbf{U}-\textbf{U}_n) (t, \cdot)\|_{L^4(\Omega; \mathbb{R}^N)} \\
			& + \sup_{t\in [0,T]} \| \langle \nu_{(t, \cdot)}; [\varrho-\overline{\varrho}]_{res} \rangle\|_{L^{\gamma}(\Omega; \mathbb{R}^N)} \| (\textbf{U}-\textbf{U}_n) (t, \cdot)\|_{L^{\frac{2\gamma}{\gamma-1}}(\Omega; \mathbb{R}^N)} \\
			& \leq C \varepsilon.
		\end{align*}
		For the last term we simply have
		\begin{equation*}
			\int_{\Omega} |\textbf{U}-\textbf{U}_n|^2 (t,\cdot) dx \leq  \sup_{t\in [0,T]} \| (\textbf{U}- \textbf{U}_n) (t,\cdot) \|_{H^M(\Omega; \mathbb{R}^N)} < \varepsilon.
		\end{equation*}
		Similarly,
		\begin{align*}
			\int_{\Omega} \langle \nu_{t, x}; P(\varrho)-P'(r)(\varrho-r) -P(r) \rangle dx &= \int_{\Omega}  \langle \nu_{t, x}; P(\varrho)-P'(r_n)(\varrho-r_n) -P(r_n) \rangle dx  \\
			&+ \int_{\Omega} \langle \nu_{t, x}; P'(r_n)(\varrho-r_n)-P'(r)(\varrho-r) \rangle dx - \int_{\Omega} \langle \nu_{t,x}; P(r)-P(r_n) \rangle dx \\
			&= \int_{\Omega} \langle \nu_{t,x}; P(r_n)-P'(r)(r_n-r)+P(r) \rangle dx \\
			&- \int_{\Omega} \langle \nu_{t, x};  [P'(r)-P'(r_n)](\varrho -r_n) \rangle dx \\
			&= \frac{P''(\xi_1)}{2} \int_{\Omega} (r-r_n)^2(t,\cdot) dx - P''(\xi_2) \int_{\Omega} \langle \nu_{t, x}; \varrho-\overline{\varrho} \rangle (r-r_n) dx \\
			&+ P''(\xi_2) \int_{\Omega} (r-r_n) (r_n-\overline{\varrho}) (t,\cdot) dx.
		\end{align*}
		We can now focus on the last two lines: the first term is simply bounded as follows
		\begin{equation*}
			\int_{\Omega} (r-r_n)^2(t,\cdot) dx \leq \sup_{t\in [0,T]} \| (r-r_n)(t,\cdot) \|_{H^M(\Omega)} <\varepsilon.
		\end{equation*}
		The second term can be rewritten as
		\begin{align*}
			\int_{\Omega} \langle \nu_{t, x}; \varrho-\overline{\varrho} \rangle  (r-r_n) (t,\cdot) dx&= \int_{\Omega} \langle \nu_{t, x}; [\varrho-\overline{\varrho}]_{ess}\rangle (r-r_n) (t,\cdot) dx+\int_{\Omega} \langle \nu_{t, x}; [\varrho-\overline{\varrho}]_{res} \rangle  (r-r_n) (t,\cdot) dx\\
			& \leq \sup_{t\in [0,T]} \| \langle \nu_{(t, \cdot)}; [\varrho-\overline{\varrho}]_{ess} \rangle\|_{L^2(\Omega; \mathbb{R}^N)} \| (r-r_n) (t, \cdot)\|_{L^2(\Omega)} \\
			& + \sup_{t\in [0,T]} \| \langle \nu_{(t, \cdot)}; [\varrho-\overline{\varrho}]_{res} \rangle\|_{L^{\gamma}(\Omega; \mathbb{R}^N)} \| (r-r_n) (t, \cdot)\|_{L^{\frac{\gamma}{\gamma-1}}(\Omega)} \\
			& \leq C \varepsilon;
		\end{align*}
		notice that, if $\gamma \in (1,2)$ we use the same argument as before while if $\gamma \in [2,\infty)$ we have to use the Sobolev embedding in the $L^p$-spaces. For the last term we can use H\"{o}lder inequality to get
		\begin{align*}
			\int_{\Omega} (r-r_n) (r_n-\overline{\varrho}) (t,\cdot) dx &\leq \sup_{t\in [0,T]} \| (r_n-\overline{\varrho})(t,\cdot) \|_{L^2(\Omega)} \| (r-r_n)(t,\cdot) \|_{L^2(\Omega)} \\
			& \leq C \sup_{t\in [0,T]} \| (r-r_n)(t,\cdot) \|_{H^M(\Omega)} \\
			& \leq C\varepsilon.
		\end{align*}
		Repeating the same steps for each term that appears in the relative energy inequality and introducing the operator
		\begin{align*}
			\mathcal{L}(\nu|r,\textbf{U})(\tau)&=a\int_{0}^{\tau}\int_{\Omega} \left\langle \nu_{t,x}; \frac{\textbf{m}}{\varrho} \cdot (\textbf{m}-\varrho \textbf{U}) \right\rangle dxdt + \mathcal{D}(\tau) \\
			&+\int_{0}^{\tau} \int_{\Omega} \langle \nu_{t,x}; \textbf{m}-\varrho\textbf{U} \rangle \cdot[\partial_t \textbf{U} + \nabla_x \textbf{U} \cdot \textbf{U}] dxdt \\
			&- \int_{0}^{\tau} \int_{\Omega} \left\langle \nu_{t,x}; 	\frac{(\textbf{m}-\varrho \textbf{U})\otimes (\varrho \textbf{U}-\textbf{m}) }{\varrho} \right\rangle: \nabla_x \textbf{U} dxdt \\
			& +\int_{0}^{\tau} \int_{\Omega} \langle \nu_{t,x}; p(\varrho)-p(r)\rangle \divv_x \textbf{U} dxdt \\
			& +\int_{0}^{\tau}\int_{\Omega} [\langle\nu_{t,x};(\varrho -r)\partial_tP'(r)+ (\textbf{m}-r\textbf{U}) \cdot  \nabla_x P'(r)\rangle dxdt \\
			&+ \int_{0}^{\tau} \int_{\Omega} \nabla_x \textbf{U}: 	d\mu_m + \int_{0}^{\tau} \int_{\Omega} \textbf{U}\cdot \nabla_x \textbf{U} \cdot d\mu_c -\int_{0}^{\tau} \int_{\Omega} \nabla_x P'(r)\cdot d\mu_c,
		\end{align*}
		we have
		\begin{equation*}
			\left[\mathcal{E}(\nu|r,\textbf{U})(t)\right]_{t=0}^{t=\tau}+ \mathcal{L}(\nu|r,\textbf{U})(\tau) \leq \left[\mathcal{E}(\nu|r_n,\textbf{U}_n)(t)\right]_{t=0}^{t=\tau}+ \mathcal{L}(\nu|r_n,\textbf{U}_n)(\tau)+C\varepsilon \leq C\varepsilon,
		\end{equation*}
		for some positive constant C, since for a test function we already proved that the relative energy inequality holds which is equivalent in saying that
		\begin{equation*}
			\left[\mathcal{E}(\nu|r_n,\textbf{U}_n)(t)\right]_{t=0}^{t=\tau}+ \mathcal{L}(\nu|r_n,\textbf{U}_n)(\tau) \leq 0.
		\end{equation*}
		By the arbitrary of $\varepsilon$ we can conclude that the relative energy inequality holds for $[r,\textbf{U}]$ as in our hypothesis.
		
		Repeating the same passages as we did in the proof of the previous theorem, we end up to the following inequality
		\begin{align*}
			\mathcal{E}(\nu|r,\textbf{U})(\tau) &+a\int_{0}^{\tau}\int_{\Omega} \left\langle \nu_{t,x}; \frac{1}{2\varrho} |\textbf{m}-\varrho\textbf{U}|^2 \right\rangle dxdt + \mathcal{D}(\tau) \\
			&\leq \int_{0}^{\tau} \int_{\Omega} \left\langle \nu_{t,x}; \left|\frac{(\textbf{m}-\varrho \textbf{U})\otimes (\varrho \textbf{U}-\textbf{m}) }{\varrho}\right| \right\rangle |\nabla_x \textbf{U}| dxdt \\
			& + \int_{0}^{\tau} \int_{\Omega} \langle \nu_{t,x}; |p(\varrho)-p'(r)(\varrho-r)-p(r)|\rangle |\divv_x \textbf{U}| dxdt \\
			&+ \int_{0}^{\tau} \int_{\Omega} |\nabla_x \textbf{U}|\cdot d|\mu_m| + \int_{0}^{\tau} \int_{\Omega} |\textbf{U}\cdot \nabla_x \textbf{U}| \cdot d|\mu_c| + \int_{0}^{\tau} \int_{\Omega} |\nabla_x P'(r)|\cdot d|\mu_c|.
		\end{align*}
		The thesis now follows as before - the only thing that changes is that in this case $\textbf{U}$ and $P(r)-P(\overline{\varrho})$ are $L^{\infty}$-functions, but still we can control the terms $|\nabla_x \textbf{U}|$, $|\divv_x \textbf{U}|$, $|\textbf{U} \cdot \nabla_x \textbf{U}|$ and $|\nabla_x P(r)|$ by some constants.
	\end{proof}

	\begin{remark}
		This theorem applies to the already know results concerning strong solutions; in particular
		\begin{itemize}
			\item[(i)] if $\Omega$ is bounded, for local in time solutions see \cite{Schochet}, and \cite{PanZhao} for the global one; 
			\item[(ii)] if $\Omega= \mathbb{R}^3$, for local in time solution see for instance \cite{Kato}, \cite{Majda}, and \cite{SidThomWang} for the global one.
		\end{itemize}
	\end{remark}

	\subsection{Vanishing viscosity limit}
	
	We conclude showing an application of the weak-strong uniqueness principle: the solutions of the Navier--Stokes system converge in the zero viscosity limit to the strong solution of the Euler system with damping on the life span of the latter.
	
	\begin{theorem} \label{Thvvl}
		Let $\Omega \subset \mathbb{R}^N$, $N=2,3$ be a domain with compact Lipschitz boundary and $\overline{\varrho} > 0$ be a given far field density if $\Omega$ is unbounded. Suppose that $\gamma > \frac{N}{2}$ and let $\varrho_R$, $\textbf{u}_R$ be a family of weak solutions to the Navier--Stokes system \eqref{continuity equation N-S} -- \eqref{boundary condition N-S} in
		\[
		(0,T) \times \Omega_R, \ \Omega_R = \Omega \cap B_R,
		\]
		with initial data $\{ \varrho_{R,0}-\overline{\varrho}, \textbf{m}_{R,0} = \varrho_{R,0} \textbf{u}_{R,0} \}_{R>0}$ such that
		\begin{equation} \label{convergence of the initial densities}
			\varrho_{R,0}-\overline{\varrho} \rightharpoonup \varrho_0 -\overline{\varrho} \quad \mbox{in } L^2+L^{\gamma}(\Omega);
		\end{equation}
		\begin{equation} \label{convergence of the initial momenta}
			\textbf{m}_{R,0} \rightharpoonup \textbf{m}_0 \quad \mbox{in } L^2+L^{\frac{2\gamma}{\gamma+1}}(\Omega; \mathbb{R}^N).
		\end{equation}
		Suppose that $\varrho_0>0$,  $\left(\varrho_0-\overline{\varrho},\frac{\textbf{m}_0}{\varrho_0}\right) \in H^M(\Omega)$, $M > \frac{N}{2} + 1$, and that $[r, \textbf{U}] \in H^M(\Omega)$ is the strong solution to the Euler system with damping with the same initial data. 
		
		Then
		\begin{equation*}
			\varrho_R- \overline{\varrho} \rightarrow r-\overline{\varrho}  \quad \mbox{in } C_{weak} ([0,T]; L^2+L^{\gamma}(\Omega)) \ \mbox{and in}\ L^1 ((0,T) \times K);
		\end{equation*}
		\begin{equation*}
		\textbf{m}_R  = \varrho_r \textbf{u}_R \rightarrow r\textbf{U} \quad \mbox{in } C_{weak} ([0,T]; L^2+L^{\frac{2\gamma}{\gamma+1}}(\Omega; \mathbb{R}^N)) \ 
		\mbox{and in } L^1 ((0,T) \times K; \mathbb{R}^N)
		\end{equation*}
		for any compact $K \subset \Omega$.
	\end{theorem}

	\begin{proof}
		Convergences \eqref{convergence of the initial densities}, \eqref{convergence of the initial momenta} follow easily from \eqref{bound on the initial energy}, repeating the same passages that we did in Section \ref{WSS}. We also proved that 
		\begin{equation*}
			\varrho_R- \overline{\varrho} \rightarrow \langle \nu_{(\cdot, \cdot)}; \varrho-\overline{\varrho}\rangle   \quad \mbox{in } C_{weak} ([0,T]; L^2+L^{\gamma}(\Omega));
		\end{equation*}
		\begin{equation*}
			\textbf{m}_R \rightarrow \langle \nu_{(\cdot, \cdot)}; \textbf{m} \rangle  \quad \mbox{in } C_{weak} ([0,T]; L^2+L^{\frac{2\gamma}{\gamma+1}}(\Omega; \mathbb{R}^N)),
		\end{equation*}
		where 
		\begin{equation*}
			\nu_{t,x}:(t,x)\in (0,T) \times \Omega \rightarrow \mathcal{P}([0,\infty)\times \mathbb{R}^N),
		\end{equation*}
		\begin{equation*}
			\nu \in L^{\infty}_{weak} ((0,T)\times \Omega; \mathcal{P}([0,\infty)\times \mathbb{R}^N)),
		\end{equation*} 
		is the Young measure associated to the sequence $\{( \varrho_R-\overline{\varrho}, \textbf{m}_R)\}_{R>0}$ and also the dissipative measure-valued solution to the Euler system with damping. Then, since
		\begin{equation*}
			\nu_{0, x}= \delta_{\varrho_0(x), \textbf{m}_0(x)} \quad \mbox{for a.e. }x\in \Omega,
		\end{equation*}
		we can apply \hyperref[theorem weak-strong uniqueness]{Theorem \ref*{theorem weak-strong uniqueness}} to get that
		\begin{equation*}
			\nu_{t, x} = \delta_{r(t,x), r\textbf{U}(t,x)} \quad \mbox{for a.e. } (t,x) \in (0,T) \times \Omega,
		\end{equation*}
		and hence we obtain the claim.
	\end{proof}

	\bigskip
	
	\centerline{\bf Acknowledgement}
	
	This work was supported by the Einstein Foundation, Berlin. The author wishes to thank Prof. Eduard Feireisl for the helpful advice and suggestions.


\newpage
	
	\begin{thebibliography} {}
		\thispagestyle{empty}
		
		\bibitem{AliBou}
		J.J. Alibert, G. Bouchitt\'e, 
		\textit{Non-uniform integrability and generalized {Y}oung measures}, {J. Convex Anal.} 4 (1997), 129-147
		
		\bibitem{FGSWW1}
		E. Feireisl, P. Gwiazda, A.
		{\'S}wierczewska-Gwiazda, E. Wiedemann,
		\textit{Dissipative measure-valued solutions to the compressible {N}avier--{S}tokes system}, {Calc. Var. Partial Differential Equations} 55 (2016), 55-141
		
		\bibitem{FeiKarPok}
		E. Feireisl, T. Karper, M. Pokorn\'{y},
		\textit{Mathematical Theory of Compressible Viscous Fluids: Analysis and Numerics}, Birkh\"{a}user-Springer, 2016
		
		\bibitem{FeiJinNov}
		E. Feireisl, B. J. Jin, A. Novotn\'{y},
		\textit{Relative Entropies, Suitable Weak Solutions, and Weak--Strong Uniqueness for the Compressible Navier--Stokes System}, Journal of Mathematical Fluid Dynamics 14 (2012), 717-730
		
		\bibitem{FeiLukMiz}
		E. Feireisl, M. Luk{\'a}{\v c}ov{\'a}-{M}edvi{\v d}ov{\'a}, H. Mizerov{\'a},
		\textit{Convergence of finite volume schemes for the {E}uler equations via dissipative measure--valued solutions},
		arxiv preprint No. 1803.08401, submitted to Foundations of Computational Mathematics
		
		\bibitem{FeiNov}
		E. Feireisl, A. Novotn\'{y},
		\textit{Singular limits in thermodynamics of viscous fluids}, Birkh\"{a}user-Verlag, Basel, 2009
		
		\bibitem{GSWW}
		P. Gwiazda, A. \'Swierczewska-Gwiazda, E.
		Wiedemann,
		\textit{Weak-strong uniqueness for measure-valued solutions of some compressible fluid models}, Nonlinearity 28 (2015), 3873--3890
		
		\bibitem{Kato}
		T. Kato,
		\textit{The Cauchy problem for quasi-linear symmetric hyperbolic systems}, Arch. Rational Mech. Anal., \textbf{58} (1975), 181-205
		
		\bibitem{Majda}
		A. Majda,
		\textit{Compressible fluid flow and systems of conservation laws in several space variables}, Applied Mathematical Sciences, New York, Springer-Verlag, 1984
		
		\bibitem{MalNecRuz}
		J. M\'{a}lek, J. Ne\v{c}as, M. Rokyta, M. Ru\v{z}i\v{c}ka,
		\textit{Weak and measure-valued solutions to evolutionary PDEs}, Applied Mathematics and Mathematical Computation, Springer-Science+Business Media B.V., 1996
		
		\bibitem{PanZhao}
		R. Pan, K. Zhao,
		\textit{The 3D compressible Euler Equations With Damping in a Bounded Domain}, Journal of Differential Equations, \textbf{246} (2009), 581-596
		
		\bibitem{Schochet}
		S. Schochet, 
		\textit{The compressible Euler equations in a bounded domain: existence of solutions and the incompressible limit}, Comm. Math. Phys., \textbf{104} (1986), 49-75
		
		\bibitem{SidThomWang}
		T. C. Sideris, B. Thomases, D. Wang,
		\textit{Long Time Behavior of Solutions to the 3D Compressible Euler Equations with Damping}, Communications in Partial Differential Equations, \textbf{28} (2003), 795-816
		
		\bibitem{Sue1}		
		F. Sueur,
		\textit{On the inviscid limit for the compressible {N}avier-{S}tokes system in an impermeable bounded domain}, J. Math. Fluid Mech., \textbf{\bf 16} (2014), 163-178
	\end{thebibliography}
\end{document}